\documentclass[a4paper,12pt]{article}

\usepackage{amsmath,amsfonts,amssymb}
\usepackage{amsthm}
\usepackage{color,multicol}

\newtheorem{lemma}{Lemma}[section]
\newtheorem{theorem}[lemma]{Theorem}
\newtheorem{corollary}[lemma]{Corollary}
\newtheorem{proposition}[lemma]{Proposition}
\newtheorem{definition}[lemma]{Definition}
\newtheorem{example}[lemma]{Example}

\newtheorem{remark}[lemma]{Remark}

\begin{document}

\title{On negative eigenvalues of two-dimensional Schr\"odinger operators with singular potentials}

\author{Martin Karuhanga\footnote{Department of Mathematics, Mbarara University of Science and Technology, Uganda. E-mail: \ mkaruhanga@must.ac.ug}\; and Eugene Shargorodsky\footnote{Department of Mathematics, King's College London, WC2R 2LS, Strand, London, UK.  E-mail: \ eugene.shargorodsky@kcl.ac.uk}}

\date{}

\maketitle

\begin{abstract}
We present upper estimates for the number of negative eigenvalues of  two-dimensional
Schr\"odinger operators with potentials generated by Ahlfors regular measures of arbitrary fractional dimension
$\alpha \in (0, 2]$. The estimates are given in terms of integrals of the potential with a logarithmic weight and of its
$L\log L$ type Orlicz norms. In the case $\alpha =1$,
our results are stronger than the known ones about Schr\"odinger operators with potentials supported by Lipschitz curves.
\end{abstract}
\noindent
{\bf Keywords}: Negative eigenvalues; Schr\"odinger operators; Singular potentials.
\section{Introduction}
Given a non-negative function $V \in L^1_{\textrm{loc}}(\mathbb{R}^d)$, consider the Schr\"odinger operator on $L^2(\mathbb{R}^d)$
\begin{equation}\label{1}
H_V := -\Delta - V, \;\;\;\;\;\;\;\; V \geq 0,
\end{equation}where $\Delta := \sum^d_{k = 1}\frac{\partial^2}{\partial x^2_k}$. This operator is defined by its quadratic form
\begin{eqnarray*}
&& \mathcal{E}_{V, \mathbb{R}^d}[u] = \int_{\mathbb{R}^d}|\nabla u(x)|^2\,dx - \int_{\mathbb{R}^d}V(x)|u(x)|^2\,dx ,\\
&& \mathrm{Dom}(\mathcal{E}_{V, \mathbb{R}^d}) = \left\{u\in W^1_2(\mathbb{R}^d)\cap L^2(\mathbb{R}^d, V(x)dx)\right\}.
\end{eqnarray*}
 Denote by $N_-(\mathcal{E}_{V, \mathbb{R}^d})$ the number of negative eigenvalues of $H_V$ counted according to their multiplicity.
 An estimate for $N_-(\mathcal{E}_{V, \mathbb{R}^d})$ in the case $d \ge 3$ is given by the celebrated Cwikel-Lieb-Rozenblum inequality:
 \begin{equation}\label{CLR}
N_-(\mathcal{E}_{V, \mathbb{R}^d})\le C_d\int_{\mathbb{R}^d}V(x)^{d/2}\,dx
\end{equation}
(see, e.g., \cite{BE, BEL, Roz} and the references therein). If $V \in L^{d/2}(\mathbb{R}^d)$, then this estimate implies that
\begin{equation}\label{O}
N_-(\mathcal{E}_{\lambda V, \mathbb{R}^d}) = O\left(\lambda^{d/2}\right) \ \mbox{ as } \ \lambda \to +\infty .
\end{equation}
The estimate is optimal in the sense that \eqref{O} implies that $V \in L^{d/2}(\mathbb{R}^d)$ (see, e.g., \cite[(127)]{RSIV}).

It is well known that \eqref{CLR} does not hold for $d = 2$. In this case, the Schr\"odinger operator
has at least one negative eigenvalue for any nonzero $V \ge 0$, and no estimate of the type
$$
N_-(\mathcal{E}_{V, \mathbb{R}^2}) \le \mbox{const} + \int_{\mathbb{R}^2} V(x)  W(x)\, dx
$$
can hold, provided the weight function $W$ is bounded in a neighborhood of at least
one point (see \cite{Grig}).
Most known upper estimates for $N_-(\mathcal{E}_{V, \mathbb{R}^2})$ involve terms of two types: integrals of $V$ with a logarithmic weight and
$L\log L$ type (or $L_p$, $p > 1$) Orlicz norms of $V$ (see \cite{Grig, LapSolo, MV, MV1, Eugene, Sol} and the references therein).
The following inequality is an example of such estimates
$$
N_-(\mathcal{E}_{V, \mathbb{R}^2}) \le 1 +  \mbox{const} \left(\int_{\mathbb{R}^2} V(x) \ln(1 + |x|)\, dx +
\|V\|_{\mathcal{B}, \mathbb{R}^2}\right)  , \ \ \
\forall V\ge 0 ,
$$
where $\|\cdot\|_{\mathcal{B}, \mathbb{R}^2}$ denotes the Orlicz norm \eqref{calB}, \eqref{Orlicz}. It was proved in \cite{Eugene}, where it
was also shown to be equivalent to the estimate conjectured in \cite{KMW} and weaker than the one obtained in \cite{Sol} (see \cite{Eugene}
for stronger estimates).
Ideally, one would like to have an optimal estimate of the type
\begin{equation}\label{ideal}
N_-(\mathcal{E}_{V, \mathbb{R}^2}) \le 1 + \Xi(V) ,
\end{equation}
where $\Xi$ is a combination of certain norms, $\Xi(\lambda V) = O(\lambda)$ as $\lambda \to +\infty$, and, most importantly,
\begin{equation}\label{O2}
N_-(\mathcal{E}_{\lambda V, \mathbb{R}^2}) = O\left(\lambda\right) \ \mbox{ as } \ \lambda \to +\infty
\end{equation}
implies that $\Xi(V) < \infty$. Unfortunately, even the strongest known estimates for $d = 2$ are not optimal in this sense
(see \cite{Eugene}). Finding an optimal estimate of type \eqref{ideal} seems to be a difficult problem.
The estimates for $N_-(\mathcal{E}_{V, \mathbb{R}^2})$ with $V$  supported by  Lipschitz curves  obtained in \cite{Kar, Eugene1}
show that \eqref{O2} may hold for singular potentials supported by lower-dimensional sets. We believe that a better
understanding of Schr\"odinger operators with such singular potentials (supported by fractal sets) might shed some additional light on the above problem.
This was the main motivation for the present work, although the results obtained here might be of some relevance to
the study of fractal antennae, apertures, screens, and transducers (see, e.g, \cite{CWH,CWHM,CWHMB, GSK,MW,WG}
and the references therein), especially in the case of impedance (Robin) boundary conditions (see \cite{HB,Ne1,Ne2,Ne3}).

In this paper, we  deal with the operator
\begin{equation}\label{2}
H_{V\mu} := -\Delta - V\mu\,,\,\,\,V \ge 0,
\end{equation}
on $L^2(\mathbb{R}^2)$, where $V\in L^1_{\textrm{loc}}(\mathbb{R}^2, \mu)$ and $\mu$ is a $\sigma$-finite positive Radon measure
on $\mathbb{R}^2$ that is Ahlfors regular of dimension $\alpha \in (0, 2]$ (see \eqref{Ahlfors}). We provide a unified treatment
of potentials locally integrable with respect to the Lebesgue measure on $\mathbb{R}^2$ ($\alpha = 2$),
potentials supported by curves ($\alpha = 1$), and potentials supported by sets of fractional dimension $\alpha \in (0, 1)\cup(1,2)$.
In the case $\alpha = 2$, we get the same estimate as in \cite[Theorem 6.1]{Eugene}, which is stronger than most other known
estimates that use isotropic norms. (Anisotropic norms like the ones used in \cite[Section 7]{Eugene} and \cite{LN}
are not available in the case $\alpha < 2$ and hence are not treated here.)
In the case  $\alpha = 1$, our Theorem \ref{mainthm} and Corollary \ref{maincor} are stronger than the results obtained
in \cite{Kar} and \cite{Eugene1}  as we are now able to cover Ahlfors regular curves rather than just Lipschitz ones.
In the case  $\alpha \in (0, 1)\cup(1,2)$, our results seem to be completely new.

The proof of our main result, Theorem \ref{mainthm},
follows the same blueprint as in \cite{Sol} and \cite{Eugene}, but dealing with measures supported by sets of fractional dimension
causes quite a few difficulties. Some of them are listed below.\\
1) One of the key technical ingredients in \cite{Sol} (and in \cite{Eugene}) was a result saying that the
Orlicz norm of the potential over a square of the side length $t > 0$ with a fixed centre is a continuous function of $t$.
This is no longer true for potentials of the form $V\mu$ (see \eqref{2}) if the measure $\mu$ is supported by an $\alpha$-dimensional
set with $\alpha \in (0, 1]$ and hence can charge the sides of the square. Lemma \ref{direction} allows one to choose the directions of the sides
of the square in such a way that this difficulty is avoided (see Lemma \ref{measlemma2}).\\
2) The Birman-Laptev-Solomyak method (see Section \ref{variational}) used in this paper (and in \cite{Sol}, \cite{Eugene}) splits the problem
into the radial and non-radial parts. The former is essentially a one-dimensional problem and is usually easier to handle than the latter.
If the measure $\mu$ is supported by an $\alpha$-dimensional set with $\alpha \in (0, 2)$, then the radial operator corresponding to \eqref{2}
is a one-dimensional Schr\"odinger operator whose potential is a measure that may be supported by a set of a fractional dimension
and may even have atoms if $\alpha \in (0, 1]$. Hence one needs to extend to such operators appropriate estimates known for
Schr\"odinger operators with potentials locally integrable with respect to the one-dimensional Lebesgue measure (\cite{Sol2}). This
has been carried out in \cite{KS}. \\
3) The Birman-Laptev-Solomyak method allows one to obtain spectral estimates for the non-radial part of the problem mentioned above by
splitting $\mathbb{R}^2\setminus\{0\}$ into homothetic annuli centred at $0$, getting an estimate for one of those annuli, and then extending it
by scaling to all other ones. Getting an estimate for an annulus usually involves covering it by carefully chosen squares, and
an additional difficulty in the case of operator \eqref{2} is that one has to distinguish between squares that are centred in the support
of the mesure $\mu$ and those that are not. Obviously, this complication does not arise in the standard case where $\mu$ is the two-dimensional
Lebesgue measure. Extending an estimate to all annuli by scaling is also not entirely trouble free for operator \eqref{2} as the measure
$\mu$ does not have to be homogeneous. Scaling leads to a change of measure, and one needs explicit information on how the constants
in the estimates depend on the underlying measure. More precisely, one needs to show that those constants depend only on $c_1/c_0$
and $\alpha$ from \eqref{Ahlfors}. Again, it is clear that this complication does not arise in the case where $\mu$ is the two-dimensional
Lebesgue measure.

The paper is organised as follows. Auxiliary results on Orlicz spaces and measures are collected in Section \ref{App}. The main results
are stated in Section \ref{mainresult}.  In Section \ref{variational}, we describe the Birman-Laptev-Solomyak method
and then apply it in Section \ref{proof} to the proof of Theorem \ref{mainthm}. Corollary \ref{maincor} is proved in Section \ref{corproof}. The
(non)optimality of our main estimate \eqref{maineqn} is discussed in Section \ref{remark}. We show that
$$
N_-(\mathcal{E}_{\lambda\, V\!\mu, \mathbb{R}^2}) = O\left(\lambda\right) \ \mbox{ as } \ \lambda \to +\infty
$$
implies that the first  sum in the right-hand side of \eqref{maineqn} is finite. Unfortunately, this is not the case for the second sum. However, we show
that the Orlicz $L\log L$ norm, the $\mathcal{B}$ norm (see \eqref{calB}) to be more precise, cannot be substituted with a weaker
Orlicz norm. Finally, we prove in Appendix some simple asymptotic results that are needed to justify the applicability of  a suitable
endpoint trace theorem (\cite[Theorem 11.8]{Maz};
see Theorem \ref{measthm2} below) in our setting (see the proof of Lemma \ref{meascor}).

\section{Auxiliary material}\label{App}
We start by recalling some notions and results from the theory of Orlicz spaces (see, e.g., \cite[Ch. 8]{Ad}, \cite{KR}, \cite{RR}).
Let $(\Omega, \Sigma, \mu)$ be a measure space and let $\Psi : [0, +\infty) \rightarrow [0, +\infty)$ be a non-decreasing function. The Orlicz class $K_{\Psi}(\Omega, \mu)$ is the set of all of measurable functions $f : \Omega \rightarrow \mathbb{C}\;( \textrm{or}\;\mathbb{R})$ such that
\begin{equation}\label{orliczeqn}
\int_{\Omega}\Psi(|f(x)|)d\mu(x) < \infty\,.
\end{equation} If $\Psi(t) = t^p,\; 1\le p < \infty$, this is just the $L^p(\Omega, \mu)$ space.
\begin{definition}
A continuous non-decreasing convex function $\Psi : [0, +\infty) \rightarrow [0, +\infty)$ is called an $N$-function if
$$
\underset{t \rightarrow 0+}\lim\frac{\Psi (t)}{t} = 0 \;\;\; \textrm{and }\;\;\;\underset{t \rightarrow \infty}\lim\frac{\Psi (t)}{t} = \infty.
$$ The function $\Phi : [0, +\infty) \rightarrow [0, +\infty)$ defined by
$$
\Phi(t) := \underset{s\geq 0}\sup\left(st - \Psi(s)\right)
$$ is called complementary to $\Psi$.
\end{definition}Examples of complementary functions include:
\begin{eqnarray}\label{calB}
&&\Psi(t) = \frac{t^p}{p},\;\;1 < p < \infty,\;\;\;\;\Phi(t) = \frac{t^q}{q}, \;\;\frac{1}{p} + \frac{1}{q} = 1, \nonumber \\
&&\mathcal{A}(s) = e^{|s|} - 1 - |s| , \ \ \ \mathcal{B}(s) = (1 + |s|) \ln(1 + |s|) - |s| , \ \ \ s \in \mathbb{R} .
\end{eqnarray}

We will use the following notation $a_+ := \max\{0, a\}$, $a \in \mathbb{R}$.
\begin{lemma}{\rm (\cite[Lemma 2.2]{Eugene})}\label{elem}
$\frac12\, s\ln_+ s \le \mathcal{B}(s) \le s + 2s\ln_+ s$, \ $\forall s \ge 0$.
\end{lemma}

\begin{definition}
An $N$-function $\Psi$ is said to satisfy the global $\Delta_2$-condition if there exists a positive constant $k$ such that for every $t \geq 0$,
\begin{equation}\label{global}
\Psi(2t)\le k\Psi(t).
\end{equation}
Similarly $\Psi$ is said to satisfy the $\Delta_2$-condition near infinity if there exists $t_0 > 0$ such that \eqref{global} holds for all $t \geq t_0$.
\end{definition}
\begin{definition}
A pair $(\Psi, \Omega)$ is called $\Delta$-regular if either $\Psi$ satisfies a global $\Delta_2$-condition, or $\Psi$ satisfies the $\Delta_2$-condition near infinity and $\mu(\Omega) < \infty$.
\end{definition}
\begin{lemma}{\rm (\cite[Lemma 8.8]{Ad})}
$K_{\Psi}(\Omega, \mu)$ is a vector space if and only if $(\Psi, \Omega)$ is $\Delta$-regular.
\end{lemma}
\begin{definition}
The \textsf{Orlicz space} $L_{\Psi}(\Omega, \mu)$ is the linear span of the Orlicz class $K_{\Psi}(\Omega, \mu)$, that is, the smallest vector space containing $K_{\Psi}(\Omega, \mu)$.
\end{definition}
Consequently, $K_{\Psi}(\Omega, \mu) = L_{\Psi}(\Omega, \mu)$ if and only if $(\Psi, \Omega)$ is $\Delta$-regular.\\\\
Let $\Phi$ and $\Psi$ be mutually complementary $N$-functions, and let $L_\Phi(\Omega,\mu)$,
$L_\Psi(\Omega, \mu)$ be the corresponding Orlicz spaces.  We will use the following
norms on $L_\Psi(\Omega, \mu)$
\begin{equation}\label{Orlicz}
\|f\|_{\Psi, \mu} = \|f\|_{\Psi, \Omega, \mu} = \sup\left\{\left|\int_\Omega f g d\mu\right| : \
\int_\Omega \Phi(|g|) d\mu \le 1\right\}
\end{equation}
and
\begin{equation}\label{Luxemburg}
\|f\|_{(\Psi, \mu)} = \|f\|_{(\Psi, \Omega, \mu)} = \inf\left\{\kappa > 0 : \
\int_\Omega \Psi\left(\frac{|f|}{\kappa}\right) d\mu \le 1\right\} .
\end{equation}
These two norms are equivalent
\begin{equation}\label{Luxemburgequiv}
\|f\|_{(\Psi, \mu)} \le \|f\|_{\Psi, \mu} \le 2 \|f\|_{(\Psi, \mu)}\, , \ \ \ \forall f \in L_\Psi(\Omega),
\end{equation}(see, e.g., \cite[(9.24)]{KR}).\\
Note that
\begin{equation}\label{LuxNormImpl}
\int_\Omega \Psi\left(\frac{|f|}{\kappa_0}\right) d\mu \le C_0, \ \ C_0 \ge 1  \ \ \Longrightarrow \ \
\|f\|_{(\Psi, \mu)} \le C_0 \kappa_0
\end{equation}
(see \cite{Eugene}). Indeed, since $\Psi$ is  convex and increasing on
$[0, +\infty)$, and $\Psi(0) = 0$, we get for any $\kappa \ge C_0 \kappa_0$,
\begin{equation}\label{LuxProof}
\int_{\Omega} \Psi\left(\frac{|f|}{\kappa}\right) d\mu \le
\int_{\Omega} \Psi\left(\frac{|f|}{C_0 \kappa_0}\right) d\mu \le
\frac{1}{C_0} \int_{\Omega} \Psi\left(\frac{|f|}{\kappa_0}\right) d\mu \le 1 .
\end{equation}
It follows from \eqref{LuxNormImpl} with $\kappa_0 = 1$ that
\begin{equation}\label{LuxNormPre}
\|f\|_{(\Psi, \mu)} \le \max\left\{1, \int_{\Omega} \Psi(|f|) d\mu\right\} .
\end{equation}
We will need the following
equivalent norm on $L_\Psi(\Omega, \mu)$ with $\mu(\Omega) < \infty$, which was introduced in
\cite{Sol}:
\begin{equation}\label{OrlAverage}
\|f\|^{\rm (av)}_{\Psi, \mu} = \|f\|^{\rm (av)}_{\Psi, \Omega, \mu} = \sup\left\{\left|\int_\Omega f g d\mu\right| : \
\int_\Omega \Phi(|g|) d\mu \le \mu(\Omega)\right\} .
\end{equation}

\begin{proposition}{\rm  (\cite[Theorem 9.3]{KR})}
For any $f\in L_{\Psi}(\Omega, \mu)$ and $g\in L_{\Phi}(\Omega, \mu)$
\begin{equation}\label{Holder}
\left|\int_{\Omega}fg\,d\mu\right| \le \|f\|_{\Psi,\Omega, \mu}\|g\|_{\Phi,\Omega, \mu}.
\end{equation}
In particular, $fg\in L^1(\Omega, \mu)$.
\end{proposition}
The above is called the H\"older inequality for Orlicz spaces.
The following is referred to as the strengthened H\"older inequality:
\begin{equation}\label{h2}
\left|\int_{\Omega}fg\,d\mu\right| \le \|f\|_{(\Psi,\Omega, \mu)}\|g\|_{\Phi,\Omega, \mu}\,,
\end{equation}
for all $f\in L_{\Psi}(\Omega, \mu)$ and $g\in L_{\Phi}(\Omega, \mu)$ (see \cite[(9.27)]{KR}).

\begin{lemma}\label{lemma7}{\rm  (\cite[Lemma 3]{Sol})}
 For any finite collection of pairwise disjoint subsets $\Omega_k$ of $\Omega$
\begin{equation}\label{bsr1}
\sum_k\|f\|^{(av)}_{\Psi,\Omega_k, \mu} \le \|f\|^{(av)}_{\Psi,\Omega, \mu}.
\end{equation}
\end{lemma}

Let
\begin{equation}\label{OrlAverage*}
\|f\|^{\rm (av),\tau}_{\Psi,\Omega, \mu}  = \sup\left\{\left|\int_\Omega f \varphi d\mu\right| : \
\int_\Omega \Phi(|\varphi|) d\mu \le \tau\mu(\Omega)\right\},\;\;\;\tau > 0 .
\end{equation}
\begin{lemma}\label{lemma8}For any $\tau_1,\;\tau_2 > 0$
\begin{equation}\label{bsr2}
\min\left\{1, \frac{\tau_2}{\tau_1}\right\} \|f\|^{\rm (av),\tau_1}_{\Psi,\Omega, \mu}
\le \|f\|^{\rm (av),\tau_2}_{\Psi,\Omega, \mu} \le \max\left\{1, \frac{\tau_2}{\tau_1}\right\} \|f\|^{\rm (av),\tau_1}_{\Psi,\Omega, \mu}.
\end{equation}
\end{lemma}
\begin{proof}
Let
$$
X_1:=\left\{\varphi\; :\;\int_{\Omega}\Phi(|\varphi|)d\mu \le \tau_1\mu(\Omega)\right\},\;\;\;X_2
:= \left\{\varphi\; :\;\int_{\Omega}\Phi(|\varphi|)d\mu \le \tau_2\mu(\Omega)\right\}.
$$
Suppose that $\tau_1\le \tau_2$. Then, it is clear that $\|f\|^{\rm (av),\tau_1}_{\Psi,\Omega, \mu}\le \|f\|^{\rm (av),\tau_2}_{\Psi,\Omega, \mu}$. Now, since $\Phi$ is convex and $\Phi(0) = 0$, then
$$
\varphi\in X_2 \Rightarrow\;\;\frac{\tau_1}{\tau_2}\varphi\in X_1\,, \;\;\;(\textrm{cf}.\,\eqref{LuxProof}).
$$ Hence,
$$
\|f\|^{\rm (av),\tau_2}_{\Psi,\Omega, \mu} = \underset{\varphi\in X_2}\sup\left|\int_{\Omega}f\varphi d\mu\right| \le \underset{\phi\in X_1}\sup\left|\int_{\Omega}f.\left(\frac{\tau_2}{\tau_1}\phi\right)d\mu\right| = \frac{\tau_2}{\tau_1}\|f\|^{\rm (av),\tau_1}_{\Psi,\Omega, \mu}.
$$
On the other hand, suppose that $\tau_1 \geq \tau_2$. Then
$$
\|f\|^{\rm (av),\tau_2}_{\Psi,\Omega, \mu}\le \|f\|^{\rm (av),\tau_1}_{\Psi,\Omega, \mu} \le \frac{\tau_1}{\tau_2}\|f\|^{\rm (av),\tau_2}_{\Psi,\Omega, \mu}.
$$Hence,
$$
\min\left\{1, \frac{\tau_2}{\tau_1}\right\} \|f\|^{\rm (av),\tau_1}_{\Psi,\Omega, \mu} \le \|f\|^{\rm (av),\tau_2}_{\Psi,\Omega, \mu}
$$ and
$$
\|f\|^{\rm (av),\tau_2}_{\Psi,\Omega, \mu} \le \max\left\{1, \frac{\tau_2}{\tau_1}\right\} \|f\|^{\rm (av),\tau_1}_{\Psi,\Omega, \mu}.
$$
\end{proof}
As a result of the above Lemma, we have the following:
\begin{corollary}\label{avequiv}{\rm (\cite[Lemma 2.1]{Eugene})}
$$
\min\{1, \mu(\Omega)\}\, \|f\|_{\Psi, \Omega, \mu} \le \|f\|^{\rm (av)}_{\Psi, \Omega, \mu}
\le \max\{1, \mu(\Omega)\}\, \|f\|_{\Psi, \Omega, \mu}.
$$
\end{corollary}

Let $(\Omega_1, \Sigma_1)$ and $(\Omega_2, \Sigma_2)$ be a pair of measurable spaces and $\xi :(\Omega_1, \Sigma_1) \to (\Omega_2, \Sigma_2)$ be an isomorphism, i.e. let $\xi$ be a bijection such that both $\xi$ and $\xi^{-1}$ are measurable. Let $\mu$ be a
 finite measure on
 $(\Omega_2, \Sigma_2)$ and $V: (\Omega_2, \Sigma_2) \to \mathbb{C}$ be a measurable function. Then $\tilde{V} := V\circ \xi$ is a measurable
 function on $(\Omega_1, \Sigma_1)$ and $\tilde{\mu} := \mu\circ \xi$,
 $$
 \tilde{\mu}(E) = \mu(\xi(E)) , \ \ \ E \in \Sigma_1
 $$
 is a mesure on $(\Omega_1, \Sigma_1)$.
For any $c > 0$ and any mutually complementary $N$-functions $\Phi$ and $\Psi$,
one gets  using \eqref{OrlAverage} and the change of variable formula (see, e.g., \cite[Lemma 5.0.1]{strok})
\begin{eqnarray}\label{scale}
\|V\|^{(\textrm{av})}_{\Psi, \Omega_2, \mu} &=&\textrm{sup}\left\{\left|\int_{\Omega_2}V f\,d\mu\right|
\;:\;\int_{\Omega_2} \Phi(|f|)\,d\mu \le \mu(\Omega_2)\right\}\nonumber\\
&=& \textrm{sup}\left\{\frac{1}{c}\left|\int_{\Omega_1}\tilde{V} g\,d(c\tilde{\mu})\right| \;:\;\int_{\Omega_1}\Phi(|g|)\,d(c\tilde{\mu}) \le
c\tilde{\mu}(\Omega_1)\right\}\nonumber\\
&=& \frac{1}{c}\left\|\tilde{V}\right\|^{(\textrm{av})}_{\Psi, \Omega_1, c\tilde{\mu}}\;.
\end{eqnarray}
Hence, by Corollary \ref{avequiv}
\begin{equation}\label{maz7}
\left\|\tilde{V}\right\|_{\Psi, \Omega_1, c\tilde{\mu}} \le
\frac{1}{\min\{1, c\tilde{\mu}(\Omega_1)\}} \left\|\tilde{V}\right\|^{(\textrm{av})}_{\Psi, \Omega_1, c\tilde{\mu}}
= \frac{c}{\min\{1, c\tilde{\mu}(\Omega_1)\}} \|V\|^{(\textrm{av})}_{\Psi, \Omega_2, \mu}\;.
\end{equation}

\begin{lemma}\label{L1Orl}
$$
\|f\|_{L_1(\Omega, \mu)} \le \Psi^{-1}(1) \|f\|^{\rm (av)}_{\Psi, \Omega, \mu} .
$$
\end{lemma}
\begin{proof}
Clearly, one only needs to consider the case $0 < \mu(\Omega) < \infty$.
Let $\mu_1 := \frac{1}{\mu(\Omega)}\, \mu$. Then $\mu_1(\Omega) = 1$, and
using \eqref{Holder},  \cite[(9.11)]{KR}, and \eqref{scale} (with $c =  \frac{1}{\mu(\Omega)}$,
$(\Omega_1, \Sigma_1) = (\Omega_2, \Sigma_2) = (\Omega, \Sigma)$,
and $\xi(x) \equiv x$), one gets
\begin{eqnarray*}
&& \int_{\Omega} |f(x)|\,d\mu(x) =  \mu(\Omega) \int_{\Omega} |f(x)|\,d\mu_1(x) \le
\mu(\Omega) \|f\|_{\Psi, \Omega, \mu_1} \|1\|_{\Phi, \Omega, \mu_1} \\
&& = \mu(\Omega) \|f\|^{\rm (av)}_{\Psi, \Omega, \mu_1} \Psi^{-1}(1)
= \mu(\Omega) \|f\|^{\rm (av)}_{\Psi, \Omega, \frac{1}{\mu(\Omega)}\, \mu} \Psi^{-1}(1)
= \|f\|^{\rm (av)}_{\Psi, \Omega, \mu} \Psi^{-1}(1) .
\end{eqnarray*}
\end{proof}

\begin{lemma}\label{avequivB}{\rm (\cite[Lemma 2.5]{Eugene})}
Let $\mu(\Omega) > 1$. Then
$$
\|f\|^{\rm (av)}_{\mathcal{B}, \Omega, \mu}
\le \|f\|_{\mathcal{B}, \Omega, \mu} + \ln\left(\frac72\, \mu(\Omega)\right)\, \|f\|_{L_1(\Omega, \mu)} .
$$
\end{lemma}

\begin{lemma}\label{direction}
Let $\mu$ be a $\sigma$-finite Borel measure on $\mathbb{R}^2$ such that $\mu(\{x\})= 0,\;\forall x\in\mathbb{R}^2$. Let
\begin{equation}\label{sim}
\Sigma := \left\{\theta \in [0, \pi)\;:\;\exists\;l_{\theta}\mbox{ such that }\mu(l_{\theta}) > 0\right\},
\end{equation}
where $l_{\theta}$ is a line in $\mathbb{R}^2$ in the direction of the vector $(\cos\theta, \sin\theta)$. Then $\Sigma$ is at most countable.
\end{lemma}
\begin{proof}
Let $$
\Sigma_N := \left\{\theta \in [0, \pi)\;:\;\exists\; \l_{\theta} \mbox{ such that } \mu(l_{\theta}\cap B(0, N)) > 0\right\},
$$
where $B(0, N)$ is the ball of radius $N\in\mathbb{N}$ centred at $0$. Then
$$
\Sigma = \underset{N\in\mathbb{N}}\cup \Sigma_N.
$$It is now enough to show that $\Sigma_N$ is at most countable for $\forall N\in\mathbb{N}$. Suppose that $\Sigma_N$ is uncountable. Then there exists a $\delta > 0$ such that
$$
\Sigma_{N,\delta} := \left\{\theta \in [0, \pi)\;:\;\exists\; \l_{\theta} \mbox{ such that }\mu(l_{\theta}\cap B(0, N)) > \delta\right\}
$$ is infinite.
Otherwise, $\Sigma_N = \underset{n\in\mathbb{N}}\cup \Sigma_{N, \frac{1}{n}}$ would have been finite or countable. Now take
distinct $\theta_1,..., \theta_k,... \in\Sigma_{N, \delta}$. Then
$$
\mu\left(l_{\theta_k}\cap B(0, N)\right) > \delta,\;\;\;\forall k\in\mathbb{N}\,.
$$
Since $l_{\theta_j}\cap l_{\theta_k} ,\;j\neq k$ contains at most one point, then $$\mu\left(\underset{j \neq k}\cup (l_{\theta_j}\cap l_{\theta_k})\right) = 0.$$
Let
$$
\tilde{l}_{\theta_k}:= l_{\theta_k}\backslash\underset{j \neq k}\cup (l_{\theta_j}\cap l_{\theta_k})\,.
$$
Then $\tilde{l}_{\theta_j}\cap\tilde{l}_{\theta_k} = \emptyset,\;j \neq k$ and $\tilde{l}_{\theta_k}\cap B(0, N) \subset B(0, N)$.
So
$$
\sum_{k\in\mathbb{N}} \mu\left( \tilde{l}_{\theta_k}\cap B(0, N)\right) =
\mu\left(\underset{k\in\mathbb{N}}\cup (\tilde{l}_{\theta_k}\cap B(0, N))\right) \le \mu\left(B(0, N)\right) < \infty\,.
$$
But
$$
\mu\left(\tilde{l}_{\theta_k}\cap B(0, N)\right) = \mu\left(l_{\theta_k}\cap B(0, N)\right) \ge \delta ,
$$
which implies
$$
\sum_{k\in\mathbb{N}} \mu\left( \tilde{l}_{\theta_k}\cap B(0, N)\right) \geq \underset{k\in\mathbb{N}}\sum\delta = \infty\,.
$$
This contradiction means that $\Sigma_N$ is at most countable for each $N\in\mathbb{N}$. Hence $\Sigma$ is at most countable.
\end{proof}

\begin{corollary}\label{cor-direct}
There exists $\theta_0 \in [0, \pi/2)$ such that $\theta_0 \notin \Sigma$ and $\theta_0 + \frac{\pi}{2} \notin \Sigma$.
\end{corollary}
\begin{proof}
The set
$$
\Sigma - \frac{\pi}{2} := \left\{ \theta - \frac{\pi}{2} \;: \theta\in \Sigma\right\}
$$
is at most countable. This implies that there exists
$$
\theta_0 \in [0, \pi/2)\setminus\left(\Sigma \cup (\Sigma - \frac{\pi}{2})\right) .
$$
Thus $\theta_0, \theta_0 + \frac{\pi}{2}\notin \Sigma$.
\end{proof}

Let $Q$ be an arbitrary unit square with its sides in the directions determined by $\theta_0$ and $\theta_0 + \frac{\pi}{2}$ in Corollary \ref{cor-direct}.  For a given $x\in\overline{Q}$ and $t > 0$, let $Q_x(t)$ be the closed square centred at $x$ with sides of length $t$ parallel to those of $Q$.

\begin{lemma}[Cf. Lemma 4 in \cite{Sol}]\label{measlemma2}
Suppose that $\Psi$ satisfies the $\Delta_2$-condition (see \eqref{global}). Then for every $f \in L_{\Psi}(Q, \mu)$, the function
$t \longmapsto \mathcal{J}(t) :=\|f\|^{\textrm{(\textrm{av})}}_{\Psi, Q_x(t), \mu}$ is continuous and $\mathcal{J}(0+) = 0$.
\end{lemma}
\begin{proof}
Let $t > t_0 > 0$. Take any measurable function $g$ on $Q_x(t)$ such that
$$
\int_{Q_x(t)}\Phi(|g(x)|)\,d\mu \le \mu(Q_x(t))
$$
and consider $h_0 := \rho g$, where
$\rho = \frac{\mu(Q_x(t_0))}{\mu(Q_x(t))} \le 1$. Then
\begin{eqnarray*}
\int_{Q_x(t_0)}\Phi(|h_0|)\,d\mu  = \int_{Q_x(t_0)}\Phi(|\rho g|)\,d\mu \le \int_{Q_x(t)}\Phi(|\rho g|)\,d\mu\\ = \rho\int_{Q_x(t)}\Phi(| g|)\,d\mu \le \rho\mu(Q_x(t)) = \mu(Q_x(t_0)).
\end{eqnarray*}
Hence
\begin{eqnarray*}
0 &\le& \|f\|^{\textrm{(av)}}_{\Psi, Q_x(t), \mu} - \|f\|^{\textrm{(av)}}_{\Psi, Q_x(t_0), \mu}\\&=& \textrm{sup}\left\{ \left|\int_{Q_x(t)}fg\;d\mu\right| : \int_{Q_x(t)}\Phi(|g|)\,d\mu \le \mu(Q_x(t))\right\}\\ &-& \textrm{sup}\left\{ \left|\int_{Q_x(t_0)}fh\;d\mu\right| : \int_{Q_x(t_0)}\Phi(|h|)\,d\mu \le \mu(Q_x(t_0))\right\}\\&\le& \textrm{sup}\left\{ \left|\int_{Q_x(t)}fg\;d\mu\right| : \int_{Q_x(t)}\Phi(|g|)\,d\mu \le \mu(Q_x(t))\right\}\\
&-& \textrm{sup}\left\{\rho \left|\int_{Q_x(t_0)}fg\;d\mu\right| : \int_{Q_x(t)}\Phi(|g|)\,d\mu \le \mu(Q_x(t))\right\}\\&\le& \textrm{sup}\left\{\left|\int_{Q_x(t)}fg\;d\mu\right|-  \rho\left|\int_{Q_x(t_0)}fg\;d\mu\right| : \int_{Q_x(t)}\Phi(|g(x)|)\,d\mu \le \mu(Q_x(t))\right\}\\&\le&\textrm{sup}\left\{\left|\int_{Q_x(t)\setminus Q_x(t_0)}fg\;d\mu\right| : \int_{Q_x(t)}\Phi(|g(x)|)\,d\mu \le \mu(Q_x(t))\right\} \\&+& (1 - \rho)\textrm{sup}\left\{\left|\int_{Q_x(t_0)}fg\;d\mu\right| : \int_{Q_x(t)}\Phi(|g(x)|)\,d\mu \le \mu(Q_x(t))\right\}.
\end{eqnarray*}

For every interval $I\subseteq Q$ parallel to the sides of $Q$, $\mu(I) = 0$.
Then $\mu \left(Q_x(t)\setminus Q_x(t_0)\right) \longrightarrow \mu(\partial Q_x(t_0)) = 0$ as $t \longrightarrow t_0$.

Using  the H\"older inequality (see \eqref{h2}), we get
\begin{eqnarray*}
&&\sup\left\{\left|\int_{Q_x(t)\setminus Q_x(t_0)}fg\;d\mu\right| : \ \int_{Q_x(t)}\Phi(|g(x)|)\,d\mu \le \mu(Q_x(t))\right\} \\
&&\le \sup_{\int_{Q_x(t)}\Phi(|g(x)|)\,d\mu \le \mu(Q_x(t))} \|f\|_{\left(\Psi, Q_x(t)\setminus Q_x(t_0),\mu\right)}
 \|g\|_{\Phi,Q_x(t)\setminus Q_x(t_0), \mu} \\
&&\le \|f\|_{\left(\Psi,Q_x(t)\setminus Q_x(t_0),\mu\right)} 2 \max\{1, \mu(Q_x(t))\}
\end{eqnarray*}(see \eqref{Luxemburgequiv} and \eqref{LuxNormPre}).
Since $\Psi$ satisfies the $\Delta_2$ condition, it follows from \cite[Theorems 9.4 and 10.3]{KR} that
$$
\lim_{t \longrightarrow t_0} \|f\|_{\left(\Psi,Q_x(t)\setminus Q_x(t_0),\mu\right)} = 0.
$$
Further,
$$
\rho = \frac{\mu(Q_x(t_0))}{\mu(Q_x(t))} = 1 - \frac{\mu\left(Q_x(t)\setminus Q_x(t_0)\right)}{\mu(Q_x(t))}\longrightarrow 1 \;\;\textrm{as} \;\;
t \longrightarrow t_0.
$$
Hence
$$
(1 - \rho) \sup\left\{\left|\int_{Q_x(t_0)}fg\;d\mu\right| : \int_{Q_x(t)}\Phi(|g(x)|)\,d\mu \le \mu(Q_x(t))\right\} \longrightarrow 0
$$
as $t \longrightarrow t_0$. The case $t_0 > t > 0$ is proved similarly.

Finally, the equality $\mathcal{J}(0+) = 0$ follows from \cite[Theorems 9.4 and 10.3]{KR}.
\end{proof}

We will use the following pair of mutually complementary $N$-functions
\begin{equation}\label{thepair}
\mathcal{A}(s) = e^{|s|} - 1 - |s| , \ \ \ \mathcal{B}(s) = (1 + |s|) \ln(1 + |s|) - |s| , \ \ \ s \in \mathbb{R} .
\end{equation}

\begin{definition}
 Let $\mu$ be a positive Radon measure on $\mathbb{R}^2$. We say the measure $\mu$ is Ahlfors regular of dimension $\alpha \in (0, 2]$
 if there exist positive constants $c_0$ and $c_1$ such that
\begin{equation}\label{Ahlfors}
c_0r^{\alpha} \le \mu(B(x, r)) \le c_1r^{\alpha}\;
\end{equation}
for all $0< r \le \mathrm{diam(supp}\,\mu)$ and all $x\in$ $\mathrm{supp}\, \mu$, where $B(x, r)$ is a ball of radius $r$ centred at $x$ and the constants $c_0$ and $c_1$ are independent of the balls.
\end{definition}
If the measure $\mu$ is $\alpha$-dimensional Ahlfors regular, then it is equivalent to the $\alpha$-dimensional Hausdorff measure (see, e.g., \cite[Lemma 1.2]{Dav} ).
If $\mathrm{supp}\,\mu$ is unbounded, \eqref{Ahlfors} is satisfied for all $ r > 0$.  For more details and examples of unbounded Ahlfors regular sets, see for example  \cite{Dav, HUT, STR}.

Suppose that $\mu$ is the usual one-dimensional Lebesgue measure on a horizontal or a vertical line.
Then \eqref{Ahlfors} holds with $\alpha = 1$. This implies $\mu(I) \neq 0$ for every nonempty subinterval $I$ of that line.
Hence the need of Lemma \ref{direction} and Corollary \ref{cor-direct} for the validity of Lemma \ref{measlemma2} in this case.

Throughout the paper, we consider integrals and Orlicz norms with respect to $\mu$ over closed rather than open sets.
This is because the $\mu$ measure of the boundary of a set may well be positive.

\section{The main result}\label{mainresult}

Let $\mathcal{H}$ be a Hilbert space and let $\mathbf{q}$ be a Hermitian form with a domain
$\mbox{Dom}\, (\mathbf{q}) \subseteq \mathcal{H}$. Set
\begin{equation}\label{hermitian}
N_- (\mathbf{q}) := \sup\left\{\dim \mathcal{L}\, | \  \mathbf{q}[u] < 0, \,
\forall u \in \mathcal{L}\setminus\{0\}\right\} ,
\end{equation}
where $\mathcal{L}$ denotes a linear subspace of $\mbox{Dom}\, (\mathbf{q})$.  The number $N_- (\mathbf{q})$ is called the Morse index of $\mathbf{q}$. If $\mathbf{q}$
is the quadratic form of a self-adjoint operator $A$ with no essential spectrum in $(-\infty, 0)$, then
by the variational principle,
$N_- (\mathbf{q})$ is the number of negative eigenvalues of $A$ repeated according to their
multiplicity (see, e.g., \cite[S1.3]{BerShu} or \cite[Theorem 10.2.3]{BirSol}).

Assume without loss of generality that $0\in \mbox{supp}\,\mu$ and $\mathrm{diam(supp}\,\mu) > 1$. Let
$$
J_n = [e^{2^{n - 1}}, e^{2^n}],\;\;n > 0\;\;\;J_0 := [e^{-1}, e],\;\;\;J_n = [e^{-2^{|n|}}, e^{-2^{|n|-1}}],\;\;n < 0,
$$ and
\begin{equation}\label{meaeqn4}
G_n := \int_{|x| \in J_n}|\ln|x||V(x)\,d\mu(x), \;\;\; n\neq 0,\;\;\;\; G_0 := \int_{|x| \in J_0}V(x)d\mu(x).
\end{equation}

If $\mbox{supp}\, \mu$ is bounded, there exists $m \in \mathbb{N}$ such that
$$
\left(2\frac{c_1}{c_0}\right)^{\frac{m}{\alpha}}  < \mathrm{diam(supp}\,\mu) \le \left(2\frac{c_1}{c_0}\right)^{\frac{m + 1}{\alpha}} .
$$
Then there exists  $\eta$ such that
\begin{equation}\label{etam}
1 < \eta \le \left(2\frac{c_1}{c_0}\right)^{\frac{1}{\alpha}}  \ \ \mbox{ and } \ \
 \mathrm{diam(supp}\,\mu) = \eta \left(2\frac{c_1}{c_0}\right)^{\frac{m}{\alpha}} .
\end{equation}
If $\mbox{supp}\, \mu$ is unbounded, we just take $\eta =1$. Then we set
\begin{equation}
\label{anc0c1}
Q_n := \left\{x\in\mathbb{R}^2 \;:\; \eta\left(2\frac{c_1}{c_0}\right)^{\frac{n -1}{\alpha}} \le |x| \le
\eta\left(2\frac{c_1}{c_0}\right)^{\frac{n }{\alpha}}\right\}, \; n\in\mathbb{Z}
\end{equation}
and
 \begin{equation}\label{Dn}
 \mathcal{D}_n := \|V\|^{(\textrm{av})}_{\mathcal{B}, Q_n, \mu}
\end{equation}
(see \eqref{thepair}).

 Define the operator \eqref{2} by its quadratic form
\begin{eqnarray*}
&& \mathcal{E}_{V\mu,\mathbb{R}^2}[w] := \int_{\mathbb{R}^2}|\nabla w (x)|^2\,dx - \int_{\mathbb{R}^2}V(x)|w(x)|^2\,d\mu(x)\,,\\
&& \mathrm{Dom}(\mathcal{E}_{V\mu, \mathbb{R}^2}) =  W^1_2(\mathbb{R}^2)\cap L^2(\mathbb{R}^2, Vd\mu).
\end{eqnarray*}
Let $N_-(\mathcal{E}_{V\mu, \mathbb{R}^2})$ denote the number of negative eigenvalues of \eqref{2} counted according to their multiplicities,
i.e. the Morse index of $\mathcal{E}_{V\mu, \mathbb{R}^2}$ defined by \eqref{hermitian}.
Then we have the following result.
\begin{theorem}\label{mainthm}
Let $\mu$ be a positive Radon  measure on $\mathbb{R}^2$ that is Ahlfors regular and $V\ge 0$. Then
there exist constants $A > 0$ and $c > 0$ such that
\begin{equation}\label{maineqn}
N_-(\mathcal{E}_{V\mu, \mathbb{R}^2}) \le 1 + 4 \sum_{G_n > 1/4}  \sqrt{G_n} + A\sum_{\mathcal{D}_n > c} \mathcal{D}_n\,.
\end{equation}
\end{theorem}
\begin{corollary}\label{maincor}
Under the conditions of the above theorem, there exists a constant $B > 0$ such that
\begin{equation}\label{maineqncor}
N_-(\mathcal{E}_{V\mu, \mathbb{R}^2}) \le 1 + B\left(\int_{\mathbb{R}^2} V(x) \ln(1 + |x|)\, d\mu(x) + \|V\|_{\mathcal{B}, \mathbb{R}^2,\, \mu}\right) .
\end{equation}
\end{corollary}
The proofs of the Theorem and the Corollary are given in sections \ref{proof} and \ref{corproof} respectively.

\section{The Birman-Laptev-Solomyak method}\label{variational}
Our description of the Birman-Solomyak method of estimating $N_- (\mathcal{E}_V)$ follows
\cite{BL, Eugene, Sol, Sol2}.

Let $(r, \theta)$ denote the polar coordinates in $\mathbb{R}^2$, $r\in\mathbb{R}_+,\;\theta\in[-\pi, \pi]$ and
\begin{equation}\label{radial}
w_{\mathcal{R}}(r) := \frac{1}{2\pi}\int_{-\pi}^{\pi}w(r, \theta)d\theta,\;\;\;w_{\mathcal{N}}(r, \theta) := w(r,\theta) - w_{\mathcal{R}}(r),
\end{equation}
where $w\in C(\mathbb{R}^2\setminus\{0\})$.
Then
\begin{equation}\label{fN0}
\int_{-\pi}^\pi w_{\mathcal{N}}(r, \theta)\, d\theta = 0 , \ \ \ \forall r > 0 ,
\end{equation}
and it is easy to see that
$$
\int_{\mathbb{R}^2} w_{\mathcal{R}} v_{\mathcal{N}}\, dy = 0 , \ \ \ \forall
w, v \in C_0^{\infty}\left(\mathbb{R}^2\setminus\{0\}\right) .
$$
Hence
$w \mapsto Pw := w_{\mathcal{R}}$
extends to an orthogonal projection $P : L^2\left(\mathbb{R}^2\right) \to
L^2\left(\mathbb{R}^2\right)$.

Using the representation of the gradient in polar coordinates one gets
\begin{eqnarray*}
&& \int_{\mathbb{R}^2} \nabla w_{\mathcal{R}} \nabla v_{\mathcal{N}}\, dy =
\int_{\mathbb{R}^2} \left(\frac{\partial w_{\mathcal{R}}}{\partial r}
\frac{\partial v_{\mathcal{N}}}{\partial r} + \frac1{r^2}
\frac{\partial w_{\mathcal{R}}}{\partial \theta}
\frac{\partial v_{\mathcal{N}}}{\partial \theta}\right)\, dy \\
&& = \int_{\mathbb{R}^2} \frac{\partial w_{\mathcal{R}}}{\partial r}
\frac{\partial v_{\mathcal{N}}}{\partial r}\, dy =
\int_{\mathbb{R}^2} \left(\frac{\partial w}{\partial r}\right)_{\mathcal{R}}
\left(\frac{\partial v}{\partial r}\right)_{\mathcal{N}}\, dy = 0 , \ \ \
\forall w, v \in C^\infty_0\left(\mathbb{R}^2\setminus\{0\}\right) .
\end{eqnarray*}
Hence $P : W^1_2\left(\mathbb{R}^2\right) \to W^1_2\left(\mathbb{R}^2\right)$ is also
an orthogonal projection.

Since
\begin{eqnarray*}
&& \int_{\mathbb{R}^2} |\nabla w|^2\, dx = \int_{\mathbb{R}^2} |\nabla w_{\mathcal{R}}|^2\, dx +
\int_{\mathbb{R}^2} |\nabla w_{\mathcal{N}}|^2\, dx , \\
&& \int_{\mathbb{R}^2} V |w|^2\, d\mu(x) \le 2 \int_{\mathbb{R}^2} V|w_{\mathcal{R}}|^2\, d\mu(x) +
2\int_{\mathbb{R}^2} V |w_{\mathcal{N}}|^2\, d\mu(x) ,
\end{eqnarray*} we have
\begin{equation}\label{meaeqn1}
N_-(\mathcal{E}_{V\mu, \mathbb{R}^2}) \le N_-(\mathcal{E}_{\mathcal{R},2V\mu}) + N_-(\mathcal{E}_{\mathcal{N},2V\mu})
\end{equation} where $\mathcal{E}_{\mathcal{R},2V\mu}$ and $\mathcal{E}_{\mathcal{N}, 2V\mu}$ are the restrictions of the form $\mathcal{E}_{2V\mu, \mathbb{R}^2}$ to  $PW^1_2(\mathbb{R}^2)$ and $(I -P)W^1_2(\mathbb{R}^2)$ respectively. Therefore to estimate $N_-\left(\mathcal{E}_{V\mu,\mathbb{R}^2}\right)$, it is sufficient to find estimates for $N_-\left(\mathcal{E}_{\mathcal{R},2V\mu}\right)$ and $N_-\left(\mathcal{E}_{\mathcal{N}, 2V\mu}\right)$. \\

On the space $PW^1_2(\mathbb{R}^2)$, a simple exponential change of variables reduces the problem to a one-dimensional Schr\"odinger operator, which provides an estimate for $N_-\left(\mathcal{E}_{\mathcal{R},2V\mu}\right)$ in terms of weighted $L^1$ norms of $V$ (see \eqref{meaeqn3}, \eqref{meaeqn5}). Theorem \ref{measthm3} shows that this estimate is optimal in a sense (see also \eqref{sqrtweak}).

On the space $(I - P)W^1_2(\mathbb{R}^2)$, one gets an estimate for $N_-\left(\mathcal{E}_{\mathcal{N},2V\mu}\right)$ in terms of
Orlicz norms of $V$ (see \eqref{meaeqn6*} and \eqref{Dn}).  The variational principle (see, e.g., \cite[Lemma 3.2]{KS}) implies that
\begin{equation}\label{variat}
N_-\left(\mathcal{E}_{\mathcal{N},2V\mu}\right) \le \sum_{n\in\mathbb{Z}}N_-\left(\mathcal{E}_{\mathcal{N},2V\mu, Q_n}\right),
\end{equation}
where $Q_n$ are the annuli defined in \eqref{anc0c1},
\begin{eqnarray*}
&& \mathcal{E}_{\mathcal{N}, 2V\mu, Q_n}[w] := \int_{Q_n}|\nabla w(x)|^2\,dx - 2\int_{Q_n}V(x)|w(x)|^2\,d\mu(x),\\
&& \mathrm{Dom}\;\left(\mathcal{E}_{\mathcal{N},2V\mu, Q_n}\right) = \left\{w\in (I - P)W^1_2(Q_n)\cap L^2\left(Q_n, Vd\mu\right)\right\}.
\end{eqnarray*}

The main reason for introducing the space $(I - P)W^1_2(\mathbb{R}^2)$ is that
\begin{equation}\label{Omegan0}
\int_{Q_n} w(x)\,dx = 0, \;\;\;\forall w\in (I - P)W^1_2(Q_n)
\end{equation}
(cf. \eqref{fN0}), which allows one to use the Poincar\'e inequality and ensures that not all terms in the right-hand side of \eqref{variat} are necessarily
greater than or equal to 1.

The Ahlfors condition \eqref{Ahlfors} allows one to obtain estimates for $N_-\left(\mathcal{E}_{\mathcal{N}, 2V\mu, Q_n}\right)$ from those for
$N_-\left(\mathcal{E}_{\mathcal{N}, 2V\mu, Q_1}\right)$ by scaling $x \longmapsto x\left(2\frac{c_1}{c_0}\right)^{\frac{n -1}{\alpha}}$.
So it is sufficient to find an estimate for $N_-\left(\mathcal{E}_{\mathcal{N}, 2V\mu, Q_1}\right)$.

\section{Proof of Theorem \ref{mainthm}}\label{proof}
We need to find an estimate for the right-hand side of \eqref{meaeqn1}.
We start with the first term. Let $I$ be an arbitrary interval in $\mathbb{R}_+$. Define a measure on $\mathbb{R}_+$ by
\begin{equation}\label{1dmeas}
\nu(I) := \int_{|x|\in I}V(x)\,d\mu(x).
\end{equation}
Then (see \eqref{radial})
$$
\int_{\mathbb{R}^2}|w_{\mathcal{R}}(x)|^2 V(x)\,d\mu(x) = \int_{\mathbb{R}_+}|w_R(r)|^2d\nu(r).
$$

 Let $w \in PW^1_2(\mathbb{R}^2)$, $r = e^t$, $v(t) := w(x) = w_{\mathcal{R}}(r)$  (see \eqref{radial}). Then
 $$
 \int_{\mathbb{R}^2}|\nabla w(x)|^2dx = 2\pi\int_{\mathbb{R}}|v'(t)|^2dt
 $$
 and
\begin{eqnarray*}
\int_{\mathbb{R}^2}V(x) |w(x)|^2d\mu(x) &=& \int_{\mathbb{R}_+}|w_{\mathcal{R}}(r)|^2d\nu(r)= \int_{\mathbb{R}}|w_{\mathcal{R}}(e^t)|^2d\nu(e^t)\\&=&\int_{\mathbb{R}}|v(t)|^2\,d\nu(e^t).
\end{eqnarray*}
Let
\begin{equation}\label{meaeqn2}
\mathcal{G}_n := \frac{1}{2\pi}\int_{\mathbf{ I}_n}|t|\,d\nu(e^t), \;\;\; n\neq 0,\;\;\;\; \mathcal{G}_0 := \frac{1}{2\pi}\int_{\mathbf{ I}_0}d\nu(e^t) ,
\end{equation}
where
$$
\mathbf{ I}_n := [2^{n - 1}, 2^n], \ n > 0 ,   \ \mathbf{ I}_0 := [-1, 1] ,  \ \
\mathbf{ I}_n := [-2^{|n|}, -2^{|n| - 1}], \ n < 0.
$$
Then
\begin{equation}\label{meaeqn3}
N_-(\mathcal{E}_{\mathcal{R}, 2\nu}) \le  1 + 7.61 \sum_{\mathcal{G}_n > 0.046}  \sqrt{\mathcal{G}_n}\,,
\end{equation}
where
\begin{eqnarray*}
&& \mathcal{E}_{\mathcal{R}, 2\nu}[v] := \int_{\mathbb{R}}|v'(t)|^2\,dt - \int_{\mathbb{R}}|v(t)|^2\,d\nu(e^t),\\
&& \mathrm{Dom}(\mathcal{E}_{\mathcal{R}, 2\nu}) = W^1_2(\mathbb{R})\cap L^2(\mathbb{R}, d\nu)
\end{eqnarray*}
(see \cite{KS}).
It follows from \eqref{meaeqn4}, \eqref{1dmeas} and \eqref{meaeqn2} that $G_n = 2\pi\mathcal{G}_n$ and thus \eqref{meaeqn3} implies
\begin{equation}\label{meaeqn5}
N_-(\mathcal{E}_{\mathcal{R}, 2V\mu}) \le  1 + 4 \sum_{G_n > 1/4}  \sqrt{G_n}.
\end{equation}

Now, it remains to find an estimate for the second term in the right-hand side of \eqref{meaeqn1} (see \eqref{meaeqn6*}). We begin by stating
some auxiliary  results.\\

Let $\varphi$ be a nonnegative increasing function on $[0, +\infty)$ such that $t\varphi(t^{-1})$ decreases and tends to zero as $t \longrightarrow\infty$. Further, suppose
\begin{equation}\label{maz2}
\int_u^{+\infty}t\sigma(t) dt \le cu\sigma(u),
\end{equation} for all $u > 0$, where
\begin{equation}\label{sigma}
\sigma(v) := v\varphi\left(\frac{1}{v}\right)
\end{equation}
and
$c$ is a positive constant.

\begin{theorem}\label{measthm2}{\rm  \cite[Theorem 11.8]{Maz}}
Let $\Psi$ and $\Phi$ be mutually complementary N-functions and let $\mu$ be a positive Radon measure on $\mathbb{R}^2$.
Let $\varphi$ be the inverse function of $t \mapsto t\Phi^{-1}(t^{-1})$ and suppose it satisfies the above conditions.
Then the best, possibly infinite, constant $A_1$ in
\begin{equation}\label{maz1}
\|w^2\|_{\Psi, \mathbb{R}^2, \mu} \le A_1\|w\|^2_{W^1_2(\mathbb{R}^2)} , \ \ \ \forall w\in W^1_2(\mathbb{R}^2)\cap C(\mathbb{R}^2)
\end{equation}
is equivalent to
\begin{equation}\label{maz3}
B_1 = \sup\left\{|\log r| \mu(B(x, r))\Phi^{-1}\left(\frac{1}{\mu(B(x, r))}\right) : \ x \in \mathbb{R}^2,\, 0 < r < \frac{1}{2}\right\} ,
\end{equation}
where $B(x, r)$ is a ball of radius $r$ centred at $x$.
\end{theorem}

Let $G\subset\mathbb{R}^2$ be a bounded set with Lipschitz boundary.
Then there exists a bounded linear operator
\begin{equation}\label{opT}
T_G: W^1_2(G) \longrightarrow W^1_2(\mathbb{R}^2)
\end{equation}
such that
\begin{eqnarray*}
&& (T_Gw)|_G = w, \;\;\;\forall w \in W^1_2(G),   \\
&& T_Gw \in W^1_2(\mathbb{R}^2)\cap C(\mathbb{R}^2), \;\;\;\forall w \in W^1_2(G)\cap C\left(\overline{G}\right)
\end{eqnarray*}
(see \cite[Ch.VI, Section 3]{Stein}).

\begin{lemma}\label{meascor}{\rm  (cf.\;\cite[Corollary 11.8/2]{Maz})}
Consider the complementary N-functions $\mathcal{B}(t)= (1 + t)\ln(1 + t) - t$ and  $\mathcal{A}(t)= e^t - 1 - t$.
Let $G\subset\mathbb{R}^2$ be a bounded set with Lipschitz boundary. If a positive Radon measure $\mu$ on $\overline{G}$ satisfies the
following estimate for some $\alpha > 0$
\begin{equation}\label{ball}
\mu(B(x, r)) \le r^{\alpha}\;,\;\;\forall x\in \mathbb{R}^2 \;\;\; \textrm{and}\;\;\;  \forall r \in \left(0, \frac12\right)\,,
\end{equation}
then the inequality
$$
\|w^2\|_{\mathcal{A}, \overline{G}, \mu} \le A_1 \|T_G\|^2 \|w\|^2_{W^1_2(G)} , \ \ \ \forall w \in W^1_2(G)\cap C(\overline{G})
$$
holds with a constant $A_1$ (see \eqref{maz1}) depending only on $\alpha$.
\end{lemma}
\begin{proof}
 First let us check that the conditions of Theorem \ref{measthm2} are satisfied.
 Let $\varrho(t) := t\mathcal{B}^{-1}\left(\frac{1}{t}\right)$ and $\frac{1}{t} = \mathcal{B}(s)$. Then $\varrho(t) = \frac{s}{\mathcal{B}(s)}$. Since
 $\frac{d}{ds}\left(\frac{\mathcal{B}(s)}{s}\right) = -\frac{1}{s^2}\ln(1 + s) + \frac{1}{s}> 0$ for $s > 0$, the fraction $\frac{s}{\mathcal{B}(s)}$ is a decreasing function of $s$. It is also clear that $\frac{s}{\mathcal{B}(s)} \longrightarrow 0$ as $s \longrightarrow\infty$. Hence $\varrho(t)$ is an increasing function of $t$ and $\varrho(t) \longrightarrow 0$ as $t \longrightarrow 0+$. Further,
 \begin{equation}\label{*}
 \varrho(t) = t\mathcal{B}^{-1}\left(\frac{1}{t}\right) = \sqrt{2t}\left(1 + o(1)\right)\;\;\textrm{as}\;\;t\longrightarrow \infty
 \end{equation}
 and
 \begin{equation}\label{**}
   \varrho(t) = t\mathcal{B}^{-1}\left(\frac{1}{t}\right) = \frac{1}{\ln\frac{1}{t}}\left(1 + o(1)\right)\;\;\textrm{as}\;\;t\longrightarrow 0
 \end{equation}
 (see \eqref{larget} and \eqref{smallt} in Appendix).

 Let $\varphi(\tau) := \varrho^{-1}(\tau)$. Then $\varphi$ is an increasing function.
 Let $x = \varrho^{-1}\left(\frac{1}{t}\right)$. Then $x$ is a decreasing function of $t$, and $t = \frac{1}{\varrho(x)}$. Hence
 $$
 t\varphi(t^{-1}) = t \varrho^{-1}\left(\frac{1}{t}\right) = \frac{x}{\varrho(x)} = \frac{1}{\mathcal{B}^{-1}\left(\frac{1}{x}\right)}
 $$
 is a decreasing function of $t$.

 For small values of $\tau$,
\begin{equation}\label{***}
\varphi(\tau)= \tau e^{-\frac{1}{\tau}}e^{O(1)}
\end{equation}
(see \eqref{taut}, \eqref{large}). Hence
$$
t\varphi(t^{-1}) = e^{-t}e^{O(1)} \longrightarrow 0 \ \mbox{ as } \ t \longrightarrow \infty
$$
and (see \eqref{sigma})
 \begin{eqnarray*}
 \int_u^{+\infty}t \sigma(t)\,dt = \int_u^{+\infty}t^2\varphi\left(\frac{1}{t}\right)\,dt = \int_u^{+\infty} te^{-t}e^{O(1)}\,dt \\
 \le e^{O(1)} \int_u^{+\infty} te^{-t}\,dt  = e^{O(1)} (u +1)e^{-u} \le 2 e^{O(1)} u e^{-u} \le \\
 \le  e^{O(1)} u^2\varphi\left(\frac{1}{u}\right)
 = e^{O(1)} u\sigma(u)\;\;\;\textrm{as}\;\;\;u\longrightarrow +\infty
 \end{eqnarray*}
 (see \eqref{***}).

For large values of $\tau$,
$$
\varphi(\tau) = \frac{\tau^2}{2}(1 + o(1))
$$
(see \eqref{*}).
 Hence
 $$
 t\sigma(t)= t^2\varphi\left(\frac{1}{t}\right) = \frac{1}{2}\left(1 + o(1)\right)\;\;\;\textrm{as}\;\;\;t \longrightarrow 0+ ,
 $$
 $$
u\sigma (u) \;\longrightarrow \frac{1}{2} \ \mbox{ and } \
\int_u^{+\infty} t \sigma(t)\,dt \longrightarrow\;\textrm{constant} \;\;\;\textrm{as}\;\;\; u\longrightarrow 0+ .
 $$
 Thus $\varphi(\tau)$ satisfies  condition \eqref{maz2} for all values of $u$.

 Extend $\mu$ to $\mathbb{R}^2$ by $\mu(E) = 0$ for
$E = \mathbb{R}^2\setminus \overline{G}$. It is easy to see that then \eqref{ball} holds for every $x \in \mathbb{R}^2$, and
one has the following estimate for the constant $B_1$ in \eqref{maz3}
 \begin{eqnarray*}
  B_1 &=& \sup\left\{|\ln r|\mu(B(x, r))\mathcal{B}^{-1}\left(\frac{1}{\mu(B(x, r))}\right) | \ 0 < r < \frac{1}{2}\right\} \\
  &=& \sup_{0 < r < \frac{1}{2}}\frac{|\ln r|}{|\ln \mu(B(x, r))|}\left( 1 + o(1)\right) \le \textrm{const}\sup\frac{|\ln r|}{|\ln r^{\alpha}|}
  = \frac{\textrm{const}}{\alpha}
 \end{eqnarray*}
 (see \eqref{**} and \eqref{ball}).
 Thus one can take $A_1 \sim \frac{1}{\alpha}$ in \eqref{maz1}. It follows from Theorem \ref{measthm2} that
$$%
\|w^2\|_{\mathcal{A}, \overline{G}, \mu} = \|(T_Gw)^2\|_{\mathcal{A}, \mathbb{R}^2, \mu} \le A_1\|T_Gw\|^2_{W^1_2(\mathbb{R}^2)} \le
A_1 \|T_G\|^2 \|w\|^2_{W^1_2(G)}
$$
for all $w \in W^1_2(G)\cap C\left(\overline{G}\right)$.
 \end{proof}

 We will use the following notation:
\begin{equation}\label{ave}
w_E := \frac{1}{|E|}\int_E w(x)\,dx\,,
\end{equation}
where $E\subset\mathbb{R}^2$ is a set of a finite Lebesgue measure $|E|$.

\begin{lemma}\label{measlemma3*}
Let $G\subset\mathbb{R}^2$ be a bounded set with Lipschitz boundary
and
$\mu$ be a positive  Radon measure satisfying \eqref{ball}. Then there exists
a constant $A_2(G) > 0$ such that
for any $V\in L_{\mathcal{B}}(\overline{G}, \mu), \;V\geq 0$,
\begin{equation}\label{maz6}
\int_{\overline{G}} V|w(x)|^2d\mu(x) \le A_2(G)  \|V\|_{\mathcal{B}, \overline{G}, \mu}\int_{G}|\nabla w |^2dx
\end{equation}
 for all $w\in W^1_2(G)\cap C(\overline{G})$ with $w_G = 0$. One can take
 \begin{equation}\label{AG}
A_2(G) = A_1 \|T_G\|^2 \left(1 + C_G\right),
\end{equation}
where $A_1$ is the constant from Lemma \ref{meascor} and $C_G$ is the optimal constant in the Poincar\'e inequality
for $G$. In particular, in the case when $G = Q$ is a unit square with sides chosen in any direction, one can take
 \begin{equation}\label{A2pi}
A_2 = A_2(Q)= A_1 \|T_Q\|^2 \left(1 + \pi^{-2}\right),
\end{equation}
which depends only on $\alpha$.
\end{lemma}
\begin{proof}
The proof of \eqref{maz6}, \eqref{AG} follows from the H\"older inequality for Orlicz spaces (see \eqref{Holder}), Lemma \ref{meascor},
and the Poincar\'e inequality (see, e.g., \cite[Ch. IV, \S7, Sect. 2, Proposition 2]{DL2}). Formula \eqref{A2pi} follows
from the fact that the best constant in the Poincar\'e inequality equals $1/\lambda_2$, where $\lambda_2$ is the
smallest positive eigenvalue of the Neumann Laplacian (see \cite[Ch. IV, \S7, Sect. 2, Corollary 3]{DL2}) and that
the latter equals $\pi^2$ for the unit square $Q$ (see, e.g., \cite[Ch. VIII, \S2, Sect. 8, (2.398)]{DL3}).
\end{proof}

 \begin{lemma}\label{measlemma3}
Suppose $\mu$ satisfies \eqref{Ahlfors}. Let  $\Omega$ be a square centred in the support of $\mu$ with sides chosen in any direction.
Then there exists a square $\Omega_0 \subseteq \Omega$ with the same centre
such that for any $V\in L_{\mathcal{B}}\left(\overline{\Omega}, \mu\right)$,
$V \geq 0$ the following estimate hiolds
\begin{equation}\label{maz8}
\int_{\overline{\Omega}}V(y)|w(y)|^2d\mu(y)
\le A_2  \frac{c_1}{c_0}4^{\alpha} \|V\|^{(\textrm{av})}_{ \mathcal{B},\overline{\Omega}, \mu}\int_{\Omega}|\nabla w(y)|^2 dy
\end{equation}
for all $w\in W^1_2(\Omega)\cap C\left(\overline{\Omega}\right)$ with $w_{\Omega_0} = 0$ (see \eqref{ave}). Here, $A_2$ is the same
constant as in \eqref{A2pi}.
\end{lemma}
\begin{proof}
Let $R$ be the side length of $\Omega$. It is sufficient to prove \eqref{maz8} in the case $\frac{R}{2} \le \mbox{diam}(\mbox{supp}\, \mu)$.
Indeed, if $\frac{R}{2} > \mbox{diam}(\mbox{supp}\, \mu)$, then there exists a square $\Omega_1$ with the same centre as $\Omega$ and
with the side length $R_1$ such that $R_1 < R$, $\frac{R_1}{2} \le \mbox{diam}(\mbox{supp}\, \mu)$, and
$\overline{\Omega_1}\cap \mbox{supp}\, \mu = \overline{\Omega}\cap \mbox{supp}\, \mu$. Then \eqref{maz8} would follow from
a similar estimate for $\Omega_1$, since
$$
\int_{\overline{\Omega}}V(y)|w(y)|^2d\mu(y) = \int_{\overline{\Omega_1}}V(y)|w(y)|^2d\mu(y) \ \mbox{ and } \
\|V\|^{(\textrm{av})}_{ \mathcal{B},\overline{\Omega_1}, \mu} = \|V\|^{(\textrm{av})}_{ \mathcal{B},\overline{\Omega}, \mu} .
$$
Below, we show that in the case $\frac{R}{2} \le \mbox{diam}(\mbox{supp}\, \mu)$, \eqref{maz8} holds
with $\Omega_0 = \Omega$.

There exist an orthogonal matrix $U \in \mathbb{R}^{2\times 2}$ and
a vector $x_0 \in \mathbb{R}^2$ such that $\Omega = \xi(Q)$, where $\xi$ is the similarity transformation
$\xi(y) = RUy + x_0$, $y \in \mathbb{R}^2$. Let $\tilde{V} := V\circ \xi$ and $\tilde{\mu} := \mu\circ \xi$.
Take any $x \in \mbox{supp}\,\tilde{\mu}$, i.e. any $x \in \mathbb{R}^2$ such that
$\xi(x) \in \mbox{supp}\,\mu$. Since $\xi(B(x,r)) = B(\xi(x), Rr)$ for any $r > 0$,
\eqref{Ahlfors} implies
\begin{equation}\label{Rr1}
c_0 (Rr)^{\alpha} \le \tilde{\mu}(B(x, r)) = \mu\left(\xi(B(x, r))\right) = \mu\left(B(\xi(x), Rr)\right) \le c_1 (Rr)^{\alpha}
\end{equation}
for any positive $r \le \frac{1}{R}\, \mbox{diam}(\mbox{supp}\, \mu)$. It is clear that the latter restriction is not needed for
the upper estimate in \eqref{Rr1}, since $\mu\left(B(\xi(x), Rr)\right)$ does not change as $r$ increases beyond
$\frac{1}{R}\, \mbox{diam}(\mbox{supp}\, \mu)$.
If $x \in \mathbb{R}^2\setminus \mbox{supp}\,\tilde{\mu}$, then, obviously,
$$
\tilde{\mu}(B(x, r)) = 0, \ \ \ \forall r < \mbox{dist} \left(x, \mbox{supp}\,\tilde{\mu}\right).
$$
If $r \ge \mbox{dist} \left(x, \mbox{supp}\,\tilde{\mu}\right)$, then there exists $x_1 \in \mbox{supp}\,\tilde{\mu}$ such that
$|x- x_1| \le r$. Hence $B(x, r) \subset B(x_1, 2r)$, and it follows from \eqref{Rr1} that
$$
\tilde{\mu}(B(x, r)) \le \tilde{\mu}(B(x_1, 2r)) \le c_1 (2R)^\alpha r^\alpha .
$$
Let
$$
c := \frac{1}{c_1 (2R)^{\alpha}}\, .
$$
Then Lemma \ref{measlemma3*} applies to the measure $c\tilde{\mu}$. Using \eqref{maz7} and the equality
$$
\int_Q|\nabla (w \circ\xi)(x)|^2 dx = \int_{\Omega}|\nabla w(y)|^2 dy ,
$$
we get
\begin{eqnarray}\label{disk}
&&\int_{\overline{\Omega}}V(y)|w(y)|^2d\mu(y) = \frac{1}{c}\int_{\overline{Q}}V(\xi(x))|w(\xi(x))|^2d(c\mu(\xi(x)))\nonumber\\
&&= \frac{1}{c}\int_{\overline{Q}}\tilde{V}(x)|(w\circ\xi)(x)|^2d(c\tilde{\mu}(x))\nonumber\\
&&\le \frac{1}{c} A_2 \|\tilde{V}\|_{\mathcal{B}, \overline{Q}, c\tilde{\mu}}\int_Q|\nabla (w \circ\xi)(x)|^2 dx\nonumber \\
&&\le \frac{1}{c} A_2\, \frac{c}{\min\{1, c\tilde{\mu}\left(\overline{Q}\right)\}}\|V\|^{(\textrm{av})}_{\mathcal{B}, \overline{\Omega}, \mu}
\int_{\Omega}|\nabla w(y)|^2 dy .
\end{eqnarray}
But
\begin{eqnarray}\label{disk1}
\frac{1}{\min\{1, c\tilde{\mu}\left(\overline{Q}\right)\}} &=& \max \left\{1, \frac{1}{c\tilde{\mu}\left(\overline{Q}\right)}\right\}
= \max \left\{1, \frac{c_1 (2R)^{\alpha}}{\mu\left(\overline{\Omega}\right)}\right\}\nonumber\\
&\le& \max\left\{ 1, \frac{c_1 (2R)^{\alpha}}{c_0\left(\frac{R}{2}\right)^{\alpha}}\right\}
= \frac{c_1}{c_0}4^{\alpha}.
\end{eqnarray}
In the inequality above, we have used \eqref{Ahlfors} and the fact $\Omega$ contains a disk of radius $\frac{R}{2}$ centred in the support of $\mu$.
Now, \eqref{maz8} follows from \eqref{disk} and \eqref{disk1}.
\end{proof}
\begin{remark}
{\rm Estimate \eqref{maz8} may fail if $\Omega$ is not centred in the support of $\mu$ (see \cite[Example 3.2.11]{MK}).}
\end{remark}

Let $G\subset\mathbb{R}^2$ be a bounded set with Lipschitz boundary such that $\mu\left(\overline{G}\right) > 0$. Let $G_0$ be
the smallest closed square containing $G$ with sides chosen in the directions $\theta_0$ and $\theta_0 + \frac{\pi}{2}$
from Corollary \ref{cor-direct}. Since $\mu\left(\overline{G}\right) > 0$, there exist $x \in \mbox{supp}\, \mu$ such that
$x \in \overline{G} \subseteq G_0$.
Let $G_1$ be the closed square centred at $x$ with sides chosen in the same directions as for $G_0$ and the side length
twice that of $G_0$. Then $G_1 \supset G_0$. Finally, Let $G^*$ be the closed square with the same centre and  the
same directions of sides as $G_0$, and with the side length $3$ times that of $G_0$. Then
\begin{equation}\label{Gs}
\overline{G} \subseteq G_0 \subset G_1 \subset G^* .
\end{equation}
Since $G_1$ is centred in $\mbox{supp}\, \mu$, Lemma \ref{measlemma3} can be applied to it. On the other hand, an advantage of
$G^*$ is that it does not depend on the choice of $x \in \mbox{supp}\, \mu$ and is uniquely defined by $G$ once the direction
$\theta_0$ has been chosen. Hence one can define the following quantity
$$
\kappa_0(G) := \frac{\mu(G^*)}{\mu\left(\overline{G}\right)}\,.
$$

Further, let
  \begin{eqnarray*}
V_{*}(x) := \left\{\begin{array}{l}
  V(x), \ \;\;\mbox{ if } x\in \overline{G},   \\ \\
   0, \ \;\;\mbox{ if } x\notin \overline{G}.
\end{array}\right.
\end{eqnarray*}
Then
\begin{equation}\label{exteqn}
\|V_{*}\|^{(av)}_{\mathcal{B}, G_1, \mu} \le \|V_{*}\|^{(av)}_{\mathcal{B}, G^*, \mu}
= \|V\|^{(av), \kappa_0(G)}_{\mathcal{B}, \overline{G}, \mu} \le \kappa_0(G) \|V\|^{(av)}_{\mathcal{B}, \overline{G}, \mu}
\end{equation}
(see Lemma \ref{lemma8}).

Using the Poincar\'e inequality (see, e.g., \cite[Ch. IV, \S7, Sect. 2, Proposition 2]{DL2}),
one gets the following estimate for operator \eqref{opT}
\begin{eqnarray}\label{ext1}
&& \|T_Gw\|^2_{W^1_2(G_1)} \le \|T_Gw\|^2_{W^1_2(G^*)} \le
\|T_Gw\|^2_{W^1_2(\mathbb{R}^2)}  \nonumber \\
&& \le \|T_G\|^2\|w\|^2_{W^1_2(G)} \le \|T_G\|^2 (1 + C_G)\int_G|\nabla w(x)|^2\,dx
\end{eqnarray}
for all $w\in W^1_2(G)$ with $w_G = 0$.

\begin{lemma}\label{measlemma4}
Let $\mu$ be a positive  Radon measure on $\mathbb{R}^2$ that is Ahlfors $\alpha$--regular and let
$G\subset\mathbb{R}^2$ be a bounded set with Lipschitz boundary such that $\mu\left(\overline{G}\right) > 0$.
Choose and fix a direction satisfying Corollary \ref{cor-direct}.
Further,  let $Q_x(r)$ be the square with sides of length $r > 0$ in the chosen direction centred at
$x\in \textrm{supp}\,\mu \cap \overline{G}$. Then for any $V\in L_{\mathcal{B}}(\overline{G}, \mu), \; V\geq 0$ and any
$n\in \mathbb{N}$ there exists a finite cover of $\textrm{supp}\,\mu\cap \overline{G}$  by squares
$Q_{x_k}(r_{x_k}), r_{x_k} > 0, k = 1, 2, ..., n_0$,  such that $n_0\le n$  and
\begin{equation}\label{maz9}
\int_{\overline{G}} V(x)|w(x)|^2d\mu(x) \le A_3n^{-1}\|V\|^{(av)}_{\mathcal{B}, \overline{G}, \mu}\int_G|\nabla w(x)|^2\,dx
\end{equation}
for all $w\in W^1_2(G)\cap C(\overline{G})$ with $(T_G w)_{Q_{x_k}(r_{x_k})} = 0, k = 1,..., n_0$ and $w_G = 0$, where
\begin{equation}\label{A3def}
A_3 = C_\alpha  \frac{c_1}{c_0} \|T_G\|^2(1 + C_G)  \kappa_0(G)^2
\end{equation}
and the constant $C_\alpha$ depends only on $\alpha$.
\end{lemma}
\begin{proof}
Let $N\in\mathbb{N}$ be a bound (see, e.g., \cite[Theorem 2.7]{FM}) in the Besicovitch covering Lemma (see, e.g., \cite[Ch. 1 Theorem 1.1]{Guz}).
 If $n \le \kappa_0(G) N$,  take $n_0 = 1$ and let $Q_{x_1}(r_{x_1})$ be the square $\Omega_0$ from
 Lemma \ref{measlemma3} with $\Omega = G_1$. Then
 it follows from \eqref{maz8}, \eqref{exteqn}, and \eqref{ext1} that for all
$w\in W^1_2(G)\cap C(\overline{G})$ with $(T_G w)_{Q_{x_1}(r_{x_1})} = 0$ and $w_G = 0$,
\begin{eqnarray}\label{B3}
&&\int_{\overline{G}} V(x)|w(x)|^2\,d\mu(x) = \int_{G_1} V_*(x)|T_G w(x)|^2\,d\mu(x) \nonumber \\
&& \le A_2 \frac{c_1}{c_0}4^{\alpha}\|V_*\|^{(\textrm{av})}_{\mathcal{B}, G_1, \mu}\int_{G_1}|\nabla (T_G w)(x)|^2\,dx \nonumber\\
&& \le A_2 \frac{c_1}{c_0}4^{\alpha}\kappa_0(G) N n^{-1}\|V_*\|^{(\textrm{av})}_{\mathcal{B}, G^*, \mu}
\int_{G^*}|\nabla (T_G w)(x)|^2\,dx \nonumber \\
&&\le A_2 \frac{c_1}{c_0}4^{\alpha} \kappa_0(G) N n^{-1}\kappa_0(G) \|V\|^{(\textrm{av})}_{\mathcal{B}, \overline{G}, \mu}\|T_G\|^2(1 + C_G)
\int_{G}|\nabla w(x)|^2\,dx \nonumber \\
&& = B_2n^{-1}\|V\|^{(\textrm{av})}_{\mathcal{B}, \overline{G}, \mu}\int_{G}|\nabla w(x)|^2\,dx\,,
\end{eqnarray}
where $B_2 := A_2  \frac{c_1}{c_0}4^{\alpha}\|T_G\|^2(1 + C_G)  \kappa_0(G)^2 N$.

Now assume that $n > \kappa_0(G) N$. Lemma \ref{measlemma2} implies that for any $x\in \textrm{supp}\,\mu\cap\overline{G}$,
there is a closed square $Q_x(r_x)$ centred at $x$ such that
\begin{equation}\label{cts}
 \|V_{*}\|^{(\textrm{av})}_{\mathcal{B}, Q_x(r_x), \mu} =  \kappa_0(G) N n^{-1}\|V\|^{(av)}_{\mathcal{B}, \overline{G}, \mu}.
\end{equation}
Since $\kappa_0(G) N n^{-1} < 1$, it is not difficult to see that $Q_x(r_x) \subseteq G^*$.
Consider the covering $\Xi = \{Q_x(r_x)\}$ of $\textrm{supp}\,\mu\cap\overline{G}$. According to the Besicovitch covering Lemma,
$\Xi$ has a countable or a finite subcover $\Xi'$ that can be split into $N$ subsets $\Xi'_j,\,j = 1, ..., N$ in such  a way that the closed squares in each subset are pairwise disjoint.
Applying Lemma \ref{lemma7} and  \eqref{exteqn}, one gets
\begin{eqnarray*}
\kappa_0(G) Nn^{-1}\|V\|^{(av)}_{\mathcal{B}, \overline{G}, \mu} \textrm{card}\,\Xi'_j &=&  \underset{Q_{x}(r_{x})\in\Xi'_j}\sum \|V_{*}\|^{(\textrm{av})}_{\mathcal{B}, Q_{x}(r_{x}), \mu}\le \|V_{*}\|^{(\textrm{av})}_{\mathcal{B}, G^{*}, \mu}\\
&\le& \kappa_0(G) \|V\|^{(\textrm{av})}_{\mathcal{B}, \overline{G}, \mu}\,.
\end{eqnarray*}
Hence $\textrm{card}\,\Xi'_j \le nN^{-1}$ and
$$
n_0 := \textrm{card}\,\Xi' = \sum_{j= 1}^N \textrm{card}\,\Xi'_j \le  n.
$$
Again, using \eqref{maz8}, \eqref{ext1} and \eqref{cts}, one gets for all
$w\in W^1_2(G)\cap C(\overline{G})$ with $(T_G w)_{Q_{x_k}(r_{x_k})} = 0, k = 1,..., n_0$ and $w_G = 0$,
\begin{eqnarray*}
&& \int_{\overline{G}} V(x)|w(x)|^2d\mu(x) = \int_{\textrm{supp}\,\mu \cap \overline{G}} V(x)|w(x)|^2d\mu(x) \\
&& \le \sum_{k = 1}^{n_0}\int_{Q_{x_k}(r_{x_k})} V_{*}(x)|(T_G w)(x)|^2\,d\mu(x)\\
&& \le A_2  \frac{c_1}{c_0}4^{\alpha}\sum_{k = 1}^{n_0} \|V_{*}\|^{(\textrm{av})}_{\mathcal{B}, Q_{x_k}(r_{x_k}), \mu}
\int_{Q_{x_k}(r_{x_k})}|\nabla (T_G w)(x)|^2dx\\
&& = A_2  \frac{c_1}{c_0}4^{\alpha} \kappa_0(G) N n^{-1}\|V\|^{(\textrm{av})}_{\mathcal{B}, \overline{G}, \mu}
\sum_{k = 1}^{n_0}\int_{Q_{x_k}(r_{x_k})}|\nabla (T_G w)(x)|^2\,dx\\
&& = A_2  \frac{c_1}{c_0}4^{\alpha} \kappa_0(G) N n^{-1} \|V\|^{(\textrm{av})}_{\mathcal{B}, \overline{G}, \mu}
\sum_{j = 1}^{N}\underset{Q_{x_k}(r_{x_k})\in\Xi'_j}\sum\int_{Q_{x_k}(r_{x_k})}|\nabla (T_G w)(x)|^2\,dx\\
&& \le A_2  \frac{c_1}{c_0}4^{\alpha} \kappa_0(G) N n^{-1}\|V\|^{(\textrm{av})}_{\mathcal{B}, G, \mu}
\sum_{j = 1}^{N}\int_{G^*}|\nabla (T_G w)(x)|^2\,dx\\
&& \le A_2  \frac{c_1}{c_0}4^{\alpha} \kappa_0(G) N^2 n^{-1}\|T_G\|^2(1 + C_G)\|V\|^{(\textrm{av})}_{\mathcal{B}, G, \mu}
\int_{G}|\nabla w(x)|^2\,dx \\
&& = C_1n^{-1}\|V\|^{(\textrm{av})}_{\mathcal{B}, G, \mu}\int_{G}|\nabla w(x)|^2\,dx ,
\end{eqnarray*}
where $C_1 := A_2 \frac{c_1}{c_0}4^{\alpha} \|T_G\|^2(1 + C_G)  \kappa_0(G) N^2$. It is now left to take
\begin{equation}\label{A3}
A_3 := \max\left\{B_2, C_1\right\} = A_2  4^{\alpha} N \|T_G\|^2(1 + C_G) \frac{c_1}{c_0} \kappa_0(G)
\max\left\{\kappa_0(G), N\right\}.
\end{equation}
\end{proof}

\begin{lemma}\label{G}
Let $\mu$ and $G$ be as in Lemma \ref{measlemma4}. Then
\begin{equation}\label{A4ineq}
\int_{\overline{G}} V(x)|w(x)|^2d\mu(x) \le A_4 \|V\|^{(av)}_{\mathcal{B}, \overline{G}, \mu}\int_G|\nabla w(x)|^2\,dx
\end{equation}
for all $w\in W^1_2(G)\cap C(\overline{G})$ with $w_G = 0$, where
\begin{eqnarray}\label{A4}
A_4 = 2\|T_G\|^2(1 + C_G)  \left(A_2  \frac{c_1}{c_0}4^{\alpha} + \frac{\mathcal{B}^{-1}(1)}{|G|}\right) \kappa_0(G) .
\end{eqnarray}
\end{lemma}
\begin{proof}
It follows from \eqref{ext1} that
\begin{eqnarray*}
\left|\left(T_G w\right)_{G_1} \right|^2 = \left|\frac{1}{|G_1|}\int_{G_1} (T_G w)(x)\, dx\right|^2 \le
\frac{1}{|G_1|} \|T_Gw\|^2_{L_2(G_1)} \\
\le \frac{1}{|G|}  \|T_G\|^2 (1 + C_G)\int_G|\nabla w(x)|^2\,dx .
\end{eqnarray*}
Using Lemma \ref{L1Orl}, one gets, similarly to \eqref{B3},
\begin{eqnarray*}
&&\int_{\overline{G}} V(x)|w(x)|^2\,d\mu(x) = \int_{G_1} V_*(x)|T_G w(x)|^2\,d\mu(x)  \\
&& \le 2 \int_{G_1} V_*(x)|T_G w(x) - \left(T_G w\right)_{G_1}|^2\,d\mu(x)  \\
&& \ \ \  + 2 \int_{G_1} V_*(x) \left|\left(T_G w\right)_{G_1}\right|^2\,d\mu(x) \\
&& \le 2A_2 \frac{c_1}{c_0}4^{\alpha}\|V_*\|^{(\textrm{av})}_{\mathcal{B}, G_1, \mu}\int_{G_1}|\nabla (T_G w)(x)|^2\,dx \\
&& \ \ \ +2 \mathcal{B}^{-1}(1) \|V_*\|^{(\textrm{av})}_{\mathcal{B}, G_1, \mu} \frac{1}{|G|}  \|T_G\|^2 (1 + C_G)\int_G|\nabla w(x)|^2\,dx \\
&& \le A_4 \|V\|^{(\textrm{av})}_{\mathcal{B}, \overline{G}, \mu}\int_{G}|\nabla w(x)|^2\,dx\,,
\end{eqnarray*}
where $A_4$ is given by \eqref{A4}.
\end{proof}

\begin{remark}
{\rm If $\mu$ satisfies \eqref{Ahlfors}, then the measure $\frac1{c_1 2^\alpha} \mu$ satisfies \eqref{ball} 
(cf. the proof of Lemma \ref{measlemma3}). Applying
Lemma \ref{measlemma3*} to $\frac1{c_1 2^\alpha} \mu$ and using
\eqref{maz7} (with $c =  \frac{1}{c_1 2^\alpha}$,
$\Omega_1 = \Omega_2 = \overline{G}$,
and $\xi(x) \equiv x$) one gets a version of \eqref{A4ineq} with the following constant
\begin{equation}\label{A4'}
A'_4 = \frac{A_1 \|T_G\|^2 \left(1 + C_G\right)}{\min\left\{1, \frac1{c_1 2^\alpha}\,\mu\left(\overline{G}\right)\right\}}
\end{equation}
in place of $A_4$. The terms in \eqref{A4} and in \eqref{A3def} that depend on the measure $\mu$ are $\frac{c_1}{c_0}$ and $\kappa_0(G)$.
The latter can often be estimated above by a quantity that depends only on $\frac{c_1}{c_0}$ and $\alpha$ (see Examples
\ref{exsq} and \ref{exan} below).
On the other hand, \eqref{A4'} contains the term $\frac1{c_1 2^\alpha}\,\mu\left(\overline{G}\right)$. Although \eqref{A4'} would also work for us
(see \eqref{A4'an}), we prefer to use
\eqref{A4} as it matches \eqref{A3def} better than \eqref{A4'}. }
\end{remark}

\begin{example}\label{exsq}
{\rm Let $\Omega$ be a square centred in the support of $\mu$ with sides of length $R$ chosen in any direction.  Then
the side length of $\Omega^*$ does not exceed $3\sqrt{2}\, R$, and
$$
\mu(\Omega^*) \le c_1 \left(3\sqrt{2}\, R\right)^\alpha .
$$
If $\frac{R}{2} \le \mbox{diam}(\mbox{supp}\, \mu)$, then
$$
\mu\left(\overline{\Omega}\right) \ge c_0\left(\frac{R}{2}\right)^{\alpha} \ \mbox{ and } \
\kappa_0(\Omega) = \frac{\mu(\Omega^*)}{\mu\left(\overline{G}\right)} \le \frac{c_1}{c_0}\left(6\sqrt{2}\right)^{\alpha}.
$$
If $\frac{R}{2} > \mbox{diam}(\mbox{supp}\, \mu)$, then $\mu\left(\overline{\Omega}\right) = \mu(\Omega^*)$ and
$\kappa_0(\Omega) = 1$.}
\end{example}

\begin{example}\label{exan}
{\rm Let $G$ be a circular annulus centred at a point $x$ in the support of $\mu$ with the radii $r$ and $R$ such that
$$
\frac{R}{r} \ge \left(2\frac{c_1}{c_0}\right)^{\frac{1}{\alpha}} \ \mbox{ and } \ R \le \mbox{diam}(\mbox{supp}\, \mu) .
$$
Then the side length of the square $G^*$ equals $6R$, and
\begin{eqnarray*}
&& \mu\left(\overline{G}\right) = \mu\left(\overline{B(x, R)}\right) - \mu\left(B(x, r)\right) \ge c_0 R^\alpha - c_1 r^\alpha \\
&&\ge c_0 R^\alpha - c_1 \frac12 \frac{c_0}{c_1} R^ \alpha= \frac{c_0}{2} R^\alpha ,\\
&& \mu(G^*) \le c_1 (6R)^\alpha .
\end{eqnarray*}
Hence,
\begin{equation}\label{kappaan}
\kappa_0(G) \le \frac{c_1 (6R)^\alpha}{\frac{c_0}{2} R^\alpha} = 2\,\frac{c_1}{c_0}\,  6^\alpha .
\end{equation}
Note also that
\begin{equation}\label{A4'an}
\frac{1}{\min\left\{1, \frac1{c_1 2^\alpha}\,\mu\left(\overline{G}\right)\right\}} \le 
\frac{1}{\min\left\{1, \frac{c_0}{2^{\alpha + 1}c_1}\,R^\alpha\right\}}
= \max\left\{1, 2^{\alpha + 1}\, \frac{c_1}{c_0}\, R^{-\alpha}\right\}\, .
\end{equation}
}
\end{example}

As above, let $\mu$ be a positive  Radon measure on $\mathbb{R}^2$ that is Ahlfors $\alpha$--regular and let
$G\subset\mathbb{R}^2$ be a bounded set with Lipschitz boundary such that $\mu\left(\overline{G}\right) > 0$. Let
\begin{eqnarray}\label{qformG}
&& \mathcal{E}_{2V\mu, G}[w] : = \int_{G} |\nabla w(x)|^2 dx - 2\int_{\overline{G}} V(x) |w(x)|^2 d\mu(x) ,  \\
&&  \mbox{Dom}\, (\mathcal{E}_{2V\mu, G}) =
\left\{w \in W^1_2\left(G\right)\cap L^2\left(\overline{G}, Vd\mu\right) | \ w_G = 0\right\} . \nonumber
\end{eqnarray}

\begin{lemma}\label{measlemma5}{\rm (cf. \cite[Lemma 7.7]{Eugene})}
\begin{equation}\label{meain1}
N_- (\mathcal{E}_{2V\mu, G}) \le A_5
\|V\|^{(\textrm{av})}_{\mathcal{B}, \overline{G}, \mu} + 2 , \ \ \ \forall V \ge 0 ,
\end{equation}
where $A_5 := 2A_3$ and $A_3$ is the constant in Lemma \ref{measlemma4}.
\end{lemma}
\begin{proof}
Let $n = \left[A_5\|V\|^{(\textrm{av})}_{\mathcal{B}, \overline{G}, \mu}\right]  + 1$
in Lemma \ref{measlemma4}, where $[a]$ denotes the largest integer not greater than $a$. Take any
linear subspace $\mathcal{L} \subset \mbox{Dom}\, (\mathcal{E}_{2V\mu, G})$
such that
$$
\dim \mathcal{L} > \left[A_5 \|V\|^{(\textrm{av})}_{\mathcal{B}, \overline{G}, \mu}\right] + 2 .
$$
Since $n_0 \le n$, there exists $w \in \mathcal{L}\setminus\{0\}$ such that
$w_{Q_{x_k}(r_{x_k})} = 0$, $k = 1, \dots, n_0$ and $w_{G} = 0$. Then
\begin{eqnarray*}
\mathcal{E}_{2V\mu, G}[w]  &=& \int_{G} |\nabla w( x)|^2dx - 2\int_{\overline{G}} V( x) |w( x)|^2 d\mu(x) \\
&\ge& \int_{G} |\nabla w(x)|^2 dx -
\frac{A_5 \|V\|^{(\textrm{av})}_{\mathcal{B}, \overline{G}, \mu}}{\left[A_5 \|V\|^{(\textrm{av})}_{\mathcal{B}, \overline{G}, \mu}\right] + 1}\,
\int_{G} |\nabla w(x)|^2 dx \\
&\ge& \int_{G} |\nabla w(x)|^2 dx - \int_{G} |\nabla w(x)|^2 dx = 0 .
 \end{eqnarray*}
Hence
$$
N_- (\mathcal{E}_{2V\mu, G}) \le \left[A_5
\|V\|^{(\textrm{av})}_{\mathcal{B}, \overline{G}, \mu}\right] + 2 \le
A_5 \|V\|^{(\textrm{av})}_{\mathcal{B}, \overline{G}, \mu} + 2.
$$
\end{proof}

\begin{lemma}\label{measlemma5*}
\begin{equation}\label{meain2}
N_- (\mathcal{E}_{2V\mu, G}) \le A_6 \|V\|^{(\textrm{av})}_{\mathcal{B}, \overline{G}, \mu} , \ \ \  \forall V\geq 0,
\end{equation}
where $A_6 := 2A_3 + 4A_4$, and $A_3$, $A_4$ are the constants in \eqref{A3def} and \eqref{A4} respectively.
\end{lemma}
\begin{proof}
By \eqref{A4ineq},
$$
2\int_{\overline{G}}V(x)|w(x)|^2d\mu(x) \le 2A_4\|V\|^{\textrm{(av)}}_{\mathcal{B}, \overline{G}, \mu}\int_{G}|\nabla w(x)|^2dx
$$
for all $w\in W^1_2(G)\cap C(\overline{G})$ with $w_{G} = 0$.

If $\|V\|^{(\textrm{av})}_{\mathcal{B}, \overline{G}, \mu} \le \frac{1}{2A_4}$, then $N_- (\mathcal{E}_{2V\mu, G}) = 0$.
If $\|V\|^{(\textrm{av})}_{\mathcal{B}, \overline{G}, \mu} > \frac{1}{2A_4}$, then Lemma \ref{measlemma5}  implies
$$
N_- (\mathcal{E}_{2V\mu, G}) \le A_5 \|V\|^{(\textrm{av})}_{\mathcal{B}, \overline{G}, \mu} + 2
\le  A_6\|V\|^{(\textrm{av})}_{\mathcal{B}, \overline{G}, \mu},
$$
where $A_6 = A_5 + 4A_4 = 2A_3 + 4A_4$.
\end{proof}

Assume that $0 \in \mbox{supp}\, \mu$. Let $\mathbb{Z}_\mu := \mathbb{Z}$ if $\mbox{supp}\, \mu$ is unbounded
and $\mathbb{Z}_\mu := \mathbb{Z}\cap (-\infty, m]$ if $\mbox{supp}\, \mu$ is bounded (see \eqref{etam}).

\begin{lemma}\label{measlemma6}
There exists a constant $A_8 > 0$ such that
\begin{equation}\label{meaeqn6}
N_-\left(\mathcal{E}_{\mathcal{N}, 2V\mu, Q_n}\right) \le A_8 \|V\|^{(\textrm{av})}_{\mathcal{B}, Q_n ,\mu}, \ \ \ \forall V\ge 0, \
\forall n \in \mathbb{Z}_\mu
\end{equation}
(see \eqref{qformG} and \eqref{anc0c1}).
\end{lemma}
\begin{proof}
We start with the case $n =1$. It follows from Lemma \ref{measlemma5*} and Example \ref{exan} that
\begin{equation}\label{fol}
N_-\left(\mathcal{E}_{\mathcal{N}, 2V\mu, Q_1}\right) \le A_8 \|V\|^{(\textrm{av})}_{\mathcal{B}, Q_1, \mu}\;,\;\;\;\forall V\ge 0,
\end{equation}
with
\begin{eqnarray*}
A_8 &=& 2 C_\alpha  \frac{c_1}{c_0} \|T_{Q_1}\|^2(1 + C_{Q_1})  \left(2\,\frac{c_1}{c_0}\,  6^\alpha\right)^2 \\
&& + 8\|T_{Q_1}\|^2(1 + C_{Q_1})  \left(A_2  \frac{c_1}{c_0}4^{\alpha} + \frac{\mathcal{B}^{-1}(1)}{|Q_1|}\right) 2\,\frac{c_1}{c_0}\,  6^\alpha \\
&=& 8 C_\alpha 6^{2\alpha}  \left(\frac{c_1}{c_0}\right)^3 \|T_{Q_1}\|^2(1 + C_{Q_1})  \\
&& + 16 \|T_{Q_1}\|^2(1 + C_{Q_1})  \left(A_2  \frac{c_1}{c_0}4^{\alpha} + \frac{\mathcal{B}^{-1}(1)}{|Q_1|}\right) \frac{c_1}{c_0}\,  6^\alpha .
\end{eqnarray*}
As far as the dependence on the measure $\mu$ is concerned, $A_8$ depends only on the ratio $\frac{c_1}{c_0}$.\\

Let $\xi : Q_1 \longrightarrow Q_n$ by given by $\xi(x) := x\left(2\frac{c_1}{c_0}\right)^{\frac{n -1}{\alpha}}$.  Let $\tilde{V} := V \circ \xi,\, \tilde{\mu} := \mu \circ \xi$ and $\tilde{w} := w \circ \xi$. Since $\xi(B(x,r)) = B\left(\xi(x), \left(2\frac{c_1}{c_0}\right)^{\frac{n -1}{\alpha}}r\right)$ for any $r > 0$,
$\tilde{\mu}$ satisfies the following analogue of \eqref{Ahlfors} (cf. \eqref{Rr1})
$$
\tilde{c}_0r^{\alpha} \le \tilde{\mu}(B(x, r)) \le \tilde{c}_1r^{\alpha}
$$
for all $0 < r \le$ diam(supp $\tilde{\mu}$),  where $\tilde{c}_0 := c_0\left(2\frac{c_1}{c_0}\right)^{n-1}$,
$\tilde{c}_1 := c_1\left(2\frac{c_1}{c_0}\right)^{n - 1}$, and $\frac{\tilde{c}_1}{\tilde{c}_0} = \frac{c_1}{c_0}$.
Now,
\begin{eqnarray*}
&&\int_{Q_n}|\nabla w(y)|^2 dy - 2\int_{Q_n}V(y)|w(y)|^2d\mu(y)\\&& = \int_{Q_1}|\nabla\tilde{w}(x)|^2dx - 2
\int_{Q_1}\tilde{V}(x)|\tilde{w}(x)|^2d\tilde{\mu}(x).
\end{eqnarray*}
It follows from \eqref{fol} that
\begin{eqnarray*}
N_-\left(\mathcal{E}_{\mathcal{N}, 2V\mu, Q_n}\right) = N_-\left(\mathcal{E}_{\mathcal{N}, 2\tilde{V}\tilde{\mu}, Q_1}\right) \le
A_8 \|\tilde{V}\|^{(\textrm{av})}_{\mathcal{B}, Q_1, \tilde{\mu}}\;,\;\;\;\forall \tilde{V}\ge 0.
\end{eqnarray*}
It follows from \eqref{scale} with $c = 1$ that
$\|\tilde{V}\|^{(\textrm{av})}_{\mathcal{B}, Q_1, \tilde{\mu}} = \|V\|^{(\textrm{av})}_{\mathcal{B}, Q_n, \mu}$.
Thus
$$
N_-\left(\mathcal{E}_{\mathcal{N}, 2V\mu, Q_n}\right) \le A_8 \|V\|^{(\textrm{av})}_{\mathcal{B}, Q_n, \mu} , \;\;\;\;\;\forall V\ge 0.
$$
Hence the scaling $x \longmapsto x\left(2\frac{c_1}{c_0}\right)^{\frac{n -1}{\alpha}}$ allows one to reduce the case of any $n\in\mathbb{Z}_\mu$
to the case $n =1$.
\end{proof}
We are now in position to derive an estimate for the second term in the right-hand side of \eqref{meaeqn1} from the variational principle
(see, e.g., \cite[Lemma 3.2]{KS}). Note that $\mbox{supp}\, \mu\setminus\{0\} \subseteq \cup_{n \in \mathbb{Z}_\mu} Q_n$ and $\mu(\{0\}) = 0$,
and that \eqref{Omegan0} implies
$$
w|_{Q_n} \in \mathrm{Dom}(\mathcal{E}_{V\mu,Q_n}) , \ \ \ \forall w \in \mathrm{Dom}(\mathcal{E}_{\mathcal{N}, 2V\mu}) .
$$
Hence, the
above Lemma  implies,
for any $c < \frac{1}{A_8}$,
\begin{equation}\label{meaeqn6*}
N_-\left(\mathcal{E}_{\mathcal{N}, 2V\mu}\right) \le A_8\sum_{\mathcal{D}_n > c} \mathcal{D}_n\;,\;\;\;\forall V\ge 0
\end{equation}
(see \eqref{Dn}). Thus Theorem \ref{mainthm} follows from \eqref{meaeqn1},  \eqref{meaeqn5} and \eqref{meaeqn6*}.

\section{Proof of Corollary \ref{maincor}}\label{corproof}

It is easy to see that
\begin{equation}\label{Gns}
\sum_{G_n > 1/4}  \sqrt{G_n} \le \sum_{G_n > 1/4} 2 G_n \le 2 \sum_{n \in \mathbb{Z}} G_n .
\end{equation}

Let $\mathbf{\Omega}_{-1}$  be the closed disc $\overline{B\left(0, e^{-1}\right)}$ and $\beta \in (0, \alpha)$.
Then using \eqref{h2}, \eqref{Ahlfors}, and
Fubini's theorem one gets
\begin{eqnarray*}
&& \sum_{n < 0}  G_n \le 2 \int_{|x| \le 1/e} V(x) |\ln|x||\, d\mu(x)
\le  2 \|V\|_{\mathcal{B}, \mathbf{\Omega}_{-1}, \mu}
\left\|\ln|\cdot|\right\|_{(\mathcal{A}, \mathbf{\Omega}_{-1}, \mu)} , \\
&& \int_{\mathbf{\Omega}_{-1}} \mathcal{A}\left(\beta\left|\ln|x|\right|\right)\, d\mu(x) \le
\int_{|x| \le 1/e} e^{\ln \frac{1}{|x|^\beta}}\, d\mu(x)  \le \int_{|x| \le 1} \frac{1}{|x|^\beta}\, d\mu(x)  \\
&& = \int_{|x| \le 1} \left(\beta\int_{|x|}^1 r^{-\beta -1}\, dr + 1\right)\, d\mu(x) \\
&& = \beta \int_0^1 r^{-\beta -1} \int_{|x| \le r} d\mu(x) dr + \int_{|x| \le 1} 1\, d\mu(x) \\
&& = \beta \int_0^1 r^{-\beta -1} \mu(B(0, r))\, dr + \mu(B(0, 1)) \le \beta \int_0^1 r^{-\beta -1} c_1 r^\alpha dr +  c_1 \\
&& =  c_1 \left(\frac{\beta}{\alpha - \beta} + 1\right) = c_1\, \frac{\alpha}{\alpha - \beta} =: A_9
\end{eqnarray*}
(We have $\sum_{n < 0}  G_n \le 2 \int_{|x| \le 1/e} \cdots$ rather than $\sum_{n < 0}  G_n = \int_{|x| \le 1/e} \cdots$
in the first inequality above because $G_n$ are integrals over domains with intersections that may have positive
measure $\mu$ (see \eqref{meaeqn4}):
$$
\mu\left(\left\{x \in \mathbb{R}^2 | \ |x| \in J_{n - 1}\right\}\cap \left\{x \in \mathbb{R}^2 | \ |x| \in J_n \right\}\right) =
\mu\left(\left\{x \in \mathbb{R}^2 | \ |x| = e^{-2^{|n|}}\right\}\right)
$$
may be positive. A similar situation occurs in \eqref{Dns0} and in the proof of Lemma \ref{reversetr} below.)
Hence
$$
\left\|\ln|\cdot|\right\|_{(\mathcal{A}, \mathbf{\Omega}_{-1}, \mu)} \le \frac{1}{\beta}\, \max\{1, A_9\} =: A_{10}
$$
(see \eqref{LuxNormImpl}) and
\begin{equation}\label{A0B0}
\sum_{n < 0}  G_n \le 2 A_{10} \|V\|_{\mathcal{B}, \mathbf{\Omega}_{-1}, \mu}
\le 2 A_{10} \|V\|_{\mathcal{B}, \mathbb{R}^2, \mu} .
\end{equation}
Further,
\begin{eqnarray}\label{G0}
G_0 &=& \int_{e^{-1} \le |x| \le e} V(x)\, d\mu(x) \nonumber \\
&\le& \frac{1}{\ln\left(1 + e^{-1}\right)} \int_{e^{-1} \le |x| \le e} V(x) \ln(1 + |x|)\, d\mu(x) \nonumber \\
&\le& \frac{1}{\ln\left(1 + e^{-1}\right)} \int_{\mathbb{R}^2} V(x) \ln(1 + |x|)\, d\mu(x)
\end{eqnarray}
and
\begin{equation}\label{logcomp}
\sum_{n > 0}  G_n \le 2 \int_{|x| \ge e} V(x) \ln|x|\, d\mu(x)  \le
2 \int_{\mathbb{R}^2} V(x) \ln(1 + |x|)\, d\mu(x) .
\end{equation}
It follows from \eqref{Gns}--\eqref{logcomp} that
\begin{equation}\label{Gnsest}
\sum_{G_n > 1/4}  \sqrt{G_n} \le A_{11} \left(\int_{\mathbb{R}^2} V(x) \ln(1 + |x|)\, d\mu(x) +  \|V\|_{\mathcal{B}, \mathbb{R}^2, \mu}\right) ,
\end{equation}
where
$$
A_{11} = 2\max\left\{2A_{10}, \ \frac{1}{\ln\left(1 + e^{-1}\right)} + 2\right\} .
$$

Let $\mathbf{\Omega}_0$  be the closed unit disc $\overline{B\left(0, \eta\right)}$. It follows from Lemma \ref{lemma7} and
Corollary \ref{avequiv} that
\begin{eqnarray}\label{Dns0}
&& \sum_{n \le 0} \mathcal{D}_n = \sum_{k \le 0} \mathcal{D}_{2k} + \sum_{k \le 0} \mathcal{D}_{2k - 1} \le
2 \|V\|^{(\textrm{av})}_{\mathcal{B}, \mathbf{\Omega}_0, \mu}  \nonumber \\
&& \le 2 \max\left\{1, \mu\left(\mathbf{\Omega}_0\right)\right\} \|V\|_{\mathcal{B}, \mathbf{\Omega}_0, \mu}
\le 2 \max\left\{1, \mu\left(\mathbf{\Omega}_0\right)\right\} \|V\|_{\mathcal{B}, \mathbb{R}^2, \mu} .
\end{eqnarray}

We need the following lemma to estimate $\sum_{n \ge 1} \mathcal{D}_n$.
\begin{lemma}{\rm (cf. \cite[Lemma 8.1]{Eugene})}\label{reversetr}
There exists $A_{12} > 0$ such that
$$
\sum_{n = 1}^\infty \|V\|_{\mathcal{B}, Q_n, \mu}  \le A_{12}
\left(\|V\|_{\mathcal{B}, \mathbb{R}^2\setminus B(0, 1), \mu} +
\int_{|x| \ge 1} V(x) \ln(2 + \ln |x|)\, d\mu(x)\right)
$$
for any $V \ge 0$.
\end{lemma}
\begin{proof}
Suppose first that $\|V\|_{(\mathcal{B}, \mathbb{R}^2\setminus B(0, 1), \mu)} = 1$ and let
$$
\alpha_n := \int_{Q_n} \mathcal{B}(V(x))\, d\mu(x) , \ \ \
\kappa_n := \|V\|_{(\mathcal{B}, Q_n, \mu)} , \ \ \ n \in \mathbb{N} .
$$
Then
\begin{eqnarray*}
&&\kappa_n \le \|V\|_{(\mathcal{B}, \mathbb{R}^2\setminus B(0, 1), \mu)} = 1 , \\
&&\sum_{n = 1}^\infty \alpha_n = \sum_{n = 1}^\infty \int_{Q_n} \mathcal{B}(V(x))\, d\mu(x) \le
2\int_{\mathbb{R}^2\setminus B(0, 1)} \mathcal{B}(V(x))\, d\mu(x) = 2
\end{eqnarray*}
and it follows from Lemma \ref{elem} that
\begin{eqnarray*}
1 =  \int_{Q_n} \mathcal{B}\left(\frac{V(x)}{\kappa_n}\right) d\mu(x) \le
\int_{Q_n} \left(\frac{V(x)}{\kappa_n} +
2\frac{V(x)}{\kappa_n} \ln_+\frac{V(x)}{\kappa_n}\right) d\mu(x) \\
\le \frac1{\kappa_n} \int_{Q_n} \left(V(x) + 2 V(x) \ln_+V(x)\right) d\mu(x) +
\frac2{\kappa_n}\, \ln\frac1{\kappa_n}\, \|V\|_{L_1(Q_n, \mu)} \\
\le \frac4{\kappa_n}\, \alpha_n + \frac1{\kappa_n}\left(1 + 2 \ln\frac1{\kappa_n}\right)
\|V\|_{L_1(Q_n, \mu)} .
\end{eqnarray*}
Hence
$$
\kappa_n \le 4\alpha_n + \left(1 + 2 \ln\frac1{\kappa_n}\right)
\|V\|_{L_1(Q_n, \mu)}
$$
and
\begin{eqnarray*}
&& \sum_{n = 1}^\infty \|V\|_{\mathcal{B}, Q_n, \mu} \le 2 \sum_{n = 1}^\infty \kappa_n =
2 \sum_{\kappa_n \le 1/n^2} \kappa_n + 2 \sum_{\kappa_n > 1/n^2} \kappa_n \\
&& \le 2 \sum_{n = 1}^\infty \frac1{n^2} + 8 \sum_{n = 1}^\infty \alpha_n +
2 \sum_{n = 1}^\infty (1 + 4 \ln n)\|V\|_{L_1(Q_n, \mu)} \\
&& \le \frac{\pi^2}{3} + 16 + 2 \sum_{n = 1}^\infty (1 + 4 \ln n)
\int_{\eta\left(2\frac{c_1}{c_0}\right)^{\frac{n -1}{\alpha}} \le |x| \le \eta\left(2\frac{c_1}{c_0}\right)^{\frac{n}{\alpha}}} V(x)\, d\mu(x) \\
&& \le \frac{\pi^2}{3} + 16 + A_{13} \sum_{n = 1}^\infty
\int_{\eta\left(2\frac{c_1}{c_0}\right)^{\frac{n -1}{\alpha}} \le |x| \le \eta\left(2\frac{c_1}{c_0}\right)^{\frac{n}{\alpha}}}
V(x) \ln(2 + \ln |x|)\, d\mu(x) \\
&& \le \frac{\pi^2}{3} + 16 + 2 A_{13} \int_{|x| \ge 1} V(x) \ln(2 + \ln |x|)\, d\mu(x) \\
&& \le A_{12} \left(\|V\|_{\mathcal{B}, \mathbb{R}^2\setminus B(0, 1), \mu} +
\int_{|x| \ge 1} V(x) \ln(2 + \ln |x|)\, d\mu(x)\right)
\end{eqnarray*}
(see \eqref{Luxemburgequiv}).
The case of a general $V$ is reduced to
$\|V\|_{(\mathcal{B}, \mathbb{R}^2\setminus B(0, 1), \mu)} = 1$ by the scaling
$V \mapsto t V$, $t > 0$.
\end{proof}

Using Lemmata \ref{avequivB} and \ref{reversetr} (see also Corollary \ref{avequiv}), one gets
\begin{eqnarray*}\label{Solimpl}
&& \sum_{n \ge 1} \mathcal{D}_n = \sum_{n \ge 1} \|V\|^{(\textrm{av})}_{\mathcal{B}, Q_n, \mu}  \\
&& \le \sum_{n = 1}^\infty \|V\|_{\mathcal{B}, Q_n, \mu} +
\sum_{n = 1}^\infty \max\left\{0,\ln\left(\frac72\, \mu(Q_n)\right)\right\}\,
\int_{Q_n} V(x)\, d\mu(x)  \\
&& \le \sum_{n = 1}^\infty \|V\|_{\mathcal{B}, Q_n, \mu} +
\sum_{n = 1}^\infty \max\left\{0,\ln\left(\frac72\, c_1 \eta^\alpha\left(2\frac{c_1}{c_0}\right)^n\right)\right\}\,
\int_{Q_n} V(x)\, d\mu(x)  \\
&& \le \sum_{n = 1}^\infty \|V\|_{\mathcal{B}, Q_n, \mu} +  A_{14}
\sum_{n = 1}^\infty n \int_{\eta\left(2\frac{c_1}{c_0}\right)^{\frac{n -1}{\alpha}} \le |x| \le \eta\left(2\frac{c_1}{c_0}\right)^{\frac{n}{\alpha}}}
V(x)\, d\mu(x)  \\
&& \le \sum_{n = 1}^\infty \|V\|_{\mathcal{B}, Q_n, \mu} +  A_{15}
\sum_{n = 1}^\infty  \int_{\eta\left(2\frac{c_1}{c_0}\right)^{\frac{n -1}{\alpha}} \le |x| \le \eta\left(2\frac{c_1}{c_0}\right)^{\frac{n}{\alpha}}}
V(x)  \ln(1 + |x|)\, d\mu(x)  \\
&& \le A_{16} \Big(\|V\|_{\mathcal{B}, \mathbb{R}^2\setminus B(0, 1), \mu} +
\int_{|x| \ge 1} V(x) \ln(2 + \ln |x|)\, d\mu(x)  \\
&& \ \ \ \ \ +  \int_{|x| \ge 1}
V(x)  \ln(1 + |x|)\, d\mu(x)\Big)  \\
&& \le A_{17} \left(\|V\|_{\mathcal{B}, \mathbb{R}^2, \mu}
+ \int_{\mathbb{R}^2} V(x) \ln(1 + |x|)\, d\mu(x)\right) , \ \ \ \forall V \ge 0 .
\end{eqnarray*}
Hence it follows from \eqref{Dns0} that
\begin{equation}\label{Dnsall}
\sum_{n \in \mathbb{Z}} \mathcal{D}_n  \le A_{18} \left(\|V\|_{\mathcal{B}, \mathbb{R}^2, \mu}
+ \int_{\mathbb{R}^2} V(x) \ln(1 + |x|)\, d\mu(x)\right) .
\end{equation}
Estimate \eqref{maineqncor} now follows from Theorem \ref{mainthm} and \eqref{Gnsest}, \eqref{Dnsall}.

\section{Conluding remarks}\label{remark}
For a sequence of numbers $(a_n)_{n \in \mathbb{Z}}$, let
$$
\left\|(a_n)_{n \in \mathbb{Z}}\right\|_{1,\infty} := \sup_{s > 0}\left(s \;\textrm{card}\{n\;:\;|a_n| > s\}\right) .
$$
It is easy to see that
$$
\left\|(a_n)_{n \in \mathbb{Z}}\right\|_{1,\infty} \le \left\|(a_n)_{n \in \mathbb{Z}}\right\|_1 = \sum_{n \in \mathbb{Z}} |a_n| .
$$
Also,
\begin{equation}\label{sqrtweak}
\sum_{|a_n| > c} \sqrt{|a_n|} \le \frac{2}{\sqrt{c}}\, \left\|(a_n)_{n \in \mathbb{Z}}\right\|_{1,\infty}
\end{equation}
and
\begin{equation}\label{sqrtweak1}
\sum_{\gamma |a_n| > c} \sqrt{\gamma |a_n|} = O(\gamma) \ \mbox{ as } \ \gamma \longrightarrow + \infty \ \ \ \Longleftrightarrow \ \ \
\left\|(a_n)_{n \in \mathbb{Z}}\right\|_{1,\infty} < \infty
\end{equation}
(see \cite[(49), (77), (78)]{Eugene}).

\begin{theorem}\label{measthm3}
Let $V \ge 0$. If $N_-(\mathcal{E}_{\gamma V\mu, \mathbb{R}^2}) = O(\gamma)$ as $\gamma \longrightarrow + \infty$, then
$\left\|(G_n)_{n \in \mathbb{Z}}\right\|_{1,\infty} < \infty$.
\end{theorem}
\begin{proof}
This follows by replacing the Lebesgue measure with $\mu$ in the proofs of \cite[Theorems 9.1 and 9.2]{Eugene}.
\end{proof}
The above theorem and \eqref{sqrtweak} show that the term $\sum_{G_n > 1/4}  \sqrt{G_n}$ in \eqref{maineqn} is optimal
in a sense. Although the same cannot be said about the term $\sum_{\mathcal{D}_n > c} \mathcal{D}_n$, the following
theorem shows that it is optimal in the class of Orlicz norms. More precisely, no estimate of the type
\begin{equation}\label{type}
N_-(\mathcal{E}_{V\mu, \mathbb{R}^2}) \le \textrm{const} + \int_{\mathbb{R}^2}V(x)W(x)\,d\mu(x) + \textrm{const}\|V\|_{\Psi, \mathbb{R}^2, \mu}
\end{equation}
can hold with a norm $\|V\|_{\Psi, \mathbb{R}^2, \mu}$ weaker than $\|V\|_{\mathcal{B}, \mathbb{R}^2, \mu}$ provided the weight function $W$ is bounded in a neighbourhood of at least one point in the support of $\mu$.
\begin{theorem}\label{notype}{\rm(cf. \cite[Theorem 9.4]{Eugene})
 Let $W \ge 0$ be bounded in a neighbourhood of at least one point in the support of $\mu$ and let $\Psi$ be an N-function such that
$$
\underset{s \longrightarrow \infty}\lim\frac{\Psi(s)}{\mathcal{B}(s)} = 0.
$$
Then there exists a compactly supported $V\ge 0$ such that
$$
\int_{\mathbb{R}^2}V(x)W(x)\,d\mu(x) + \|V\|_{\Psi, \mathbb{R}^2, \mu} < \infty
$$ and $N_-(\mathcal{E}_{V\mu, \mathbb{R}^2}) = \infty$.
}
\end{theorem}
\begin{proof}
Shifting the independent variable if necessary, we can assume that $0 \in$ supp $\mu$ and $W$ is bounded
in a neighborhood of $0$. Let $r_0 > 0$ be such that $W$ is bounded in the open ball
$B(0, r_0)$.

Let
$$
\beta(s) := \sup_{t \ge s}\, \frac{\Psi(t)}{\mathcal{B}(t)}\, .
$$
Then $\beta$ is a non-increasing function, $\beta(s) \to 0$ as $s \to \infty$, and
$\Psi(s) \le \beta(s) \mathcal{B}(s)$.
Since $\Psi$ is an $N$-function, $\Psi(s)/s \to \infty$ as $s \to \infty$
(see section \ref{App}). Hence there exists $s_0 \ge e^{\frac{1}{\alpha}} > 1$ such that $\Psi(s) \ge s$ and
$\beta(s) \le 1$ for $s \ge s^{\alpha}_0$.
Choose $\rho_k \in (0, 1/s_0)$ in such a way that
$$
\sum_{k = 1}^\infty \beta\left(\frac1{\rho^{\alpha}_k}\right) < \infty .
$$
It follows from \eqref{Ahlfors} that $\forall r > 0$, the disk $B(0, r)$ contains points of the support of $\mu$ different from $0$. Let
$x^{(1)}\in \mbox{supp}\,\mu \setminus\{0\}$ be such that
$$
|x^{(1)}| < \min\left\{\frac{2}{3}\, r_0, 2\rho_1\right\} .
$$
One can choose $x^{(k)},\,k\in\mathbb{N}$ inductively as follows: suppose $x^{(1)}, ... , x^{(k)} \in$ supp$\,\mu \setminus\{0\}$ have been chosen.
Take $x^{(k + 1)} \in$ supp $\mu\setminus\{0\}$ such that
$$
|x^{(k+1)}| < \min\left\{\frac{1}{3}|x^{(k)}|, 2\rho_{k+1}\right\}.
$$
Since $|x^{(k+1)}| < \frac{1}{3}|x^{(k)}|$, it is easy to see that the open disks $B(x^{(k)}, \frac{1}{2}|x^{(k)}|)$, $k \in \mathbb{N}$ lie in $B(0, r_0)$
and are pairwise disjoint. Let $r_k := \frac{1}{2}|x^{(k)}|$. Then $r_k < \rho_k,\,k\in\mathbb{N}$. For a constant $A_{19} > 0$ to be specified later, let
\begin{eqnarray*}
& t_k := \frac{A_{19}}{\ln\frac1{r_k}}\, r_k^{-2\alpha} & \\ \\
& V(x) := \left\{\begin{array}{ll}
  t_k ,   &  x \in B\left(x^{(k)}, r_k^2\right) , \  k \in \mathbb{N}, \\
    0 ,  &   \mbox{otherwise.}
\end{array}\right. &
\end{eqnarray*}
Since the function $r \mapsto r^{\alpha}\ln\frac1{r}$ has maximum equal to $\frac{1}{\alpha e} $, one can choose  $A_{19} > 0$ such that
$A_{19}\alpha e > 1$ and
$$
t_k = \frac{A_{19}}{\ln\frac1{r_k}}\, r_k^{-2\alpha} = \frac{A_{19}}{r_k^{\alpha} \ln\frac1{r_k}}\, r_k^{-\alpha} > \frac1{r^{\alpha}_k}
> \frac1{\rho^{\alpha}_k} > s_0^\alpha \ge e .
$$
Then
\begin{eqnarray*}
&& \int_{\mathbb{R}^2}\Psi(V(x))\,d\mu(x) = \sum_{k= 1}^{\infty}\Psi(t_k)\mu\left(B(x^{(k)}, r^2_k)\right) \le
\sum_{k= 1}^{\infty}\Psi(t_k)c_1 r^{2\alpha}_{k}\\
&& \le c_1\sum_{k =1}^{\infty}r^{2\alpha}_k\beta(t_k)\mathcal{B}(t_k)
\le c_1\sum_{k =1}^{\infty}r^{2\alpha}_k\beta(t_k)(1 + t_k)\ln(1 + t_k) \\
&& < 4c_1\sum_{k =1}^{\infty}r^{2\alpha}_k\beta(t_k)t_k\ln t_k
= 4c_1\sum_{k =1}^{\infty}\beta(t_k)\frac{A_{19}}{\ln\frac{1}{r_k}}\ln\frac{A_{19}}{r^{2\alpha}_k\ln\frac{1}{r_k}} \\
&& \le 4c_1A_{19}\sum_{k =1}^{\infty}\beta\left(\frac{1}{r^{\alpha}_k}\right)\frac{1}{\ln\frac{1}{r_k}}\ln\frac{A_{19} \alpha}{r_k^{2\alpha}}\\
&& \le \textrm{const} \sum_{k = 1}^{\infty} \beta\left(\frac{1}{r^{\alpha}_k}\right) \le
\textrm{const} \sum_{k = 1}^{\infty} \beta\left(\frac{1}{\rho^{\alpha}_k}\right) < \infty.
\end{eqnarray*}
Thus $\|V\|_{\Psi, \mathbb{R}^2, \mu} < \infty$ (see \eqref{LuxNormPre} and \eqref{Luxemburgequiv}).
Since $t_k > \frac{1}{r^{\alpha}_k} > s^{\alpha}_0$, one has $t_k \le \Psi(t_k)$ and
$$
\int_{\mathbb{R}^2}V(x)\,d\mu(x) \le \int_{\mathbb{R}^2}\Psi(V(x))\,d\mu(x) < \infty.
$$
Since $W$ is bounded in $B(0, r_0)$,
$$
\int_{\mathbb{R}^2}V(x)W(x)\,d\mu(x) < \infty\,.
$$
Let
$$
w_k(x) := \left\{\begin{array}{cl}
  1 ,   & \  |x - x^{(k)}| \le r_k^2 , \\ \\
    \frac{\ln (r_k/|x - x^{(k)}|)}{\ln(1/r_k)} ,  & \ r_k^2 <  |x - x^{(k)}| \le r_k  , \\ \\
 0 , & \  |x - x^{(k)}| > r_k
\end{array}\right.
$$
(cf. \cite{Grig}). Then
$$
\int_{\mathbb{R}^2} |\nabla w_k(x)|^2\, dx = \frac{2\pi}{\ln(1/r_k)}\,.
$$
Further,
\begin{eqnarray*}
\int_{\mathbb{R}^2} V(x) |w_k(x)|^2\, d\mu(x) &\ge& \int_{B\left(x^{(k)}, r_k^2\right)} V(x)\, d\mu(x) =
t_k \mu\left(B\left(x^{(k)}, r_k^2\right)\right)\\ &\ge& t_k c_0 r^{2\alpha}_k = c_0\frac{A_{19}}{\ln \frac{1}{r_k}}\, .
\end{eqnarray*}
Hence for any $A_{19} > \frac{2\pi}{c_0}$,
$$
\mathcal{E}_{V\mu, \mathbb{R}^2}[w_k] < 0 , \ \ \ \forall k \in \mathbb{N}
$$
and $N_- (\mathcal{E}_{V\mu, \mathbb{R}^2}) = \infty$.

\end{proof}

\section{Appendix:  \ Proofs of \eqref{*}, \eqref{**} and \eqref{***}}\label{APII}\setcounter{equation}{0}
\renewcommand{\theequation}{A.\arabic{equation}}
Let $\mathcal{B}(s) = (1 + s)\ln(1 + s) - s = \frac{1}{t}$, then $s = \mathcal{B}^{-1}\left(\frac{1}{t}\right)$. For small values of $s$ (large values of $t$), using
$$
\ln(1 + s) = s - \frac{s^2}{2} + \frac{s^3}{3} + O\left(s^4\right),
$$ we have
$$
(1 + s)\ln(1 + s) - s = \frac{s^2}{2} + O\left(s^3\right) = \frac{1}{t}.
$$ One can write this in the form
\begin{eqnarray*}
\frac{s^2}{2} + s^2g(s) &=& \frac{1}{t},\;\;\;\;g(0)= 0,\\
\frac{s^2}{2}\left(1 + 2g(s)\right) &=& \frac{1}{t},\\
s\left(1 + h(s)\right) &=& \sqrt{\frac{2}{t}},\;\;\;\;h(0) = 0,
\end{eqnarray*} where $g$ and $h$ are $C^{\infty}$ smooth functions in a neighbourhood of $0$.
Let $f(s) = s\left(1 + h(s)\right)$. Then $f(0)= 0, f'(0) = 1$ and $(f^{-1})'(0) = 1$, which means
that both $f$ and $f^{-1}$ are invertible in a neighbourhood of $0$, and
$$
s = f^{-1}\left(\sqrt{\frac{2}{t}}\right) = \sqrt{\frac{2}{t}} + O\left(\frac{1}{t}\right).
$$
Thus
$$
\mathcal{B}^{-1}\left(\frac{1}{t}\right) = \sqrt{\frac{2}{t}}\left(1 + o(1)\right)  \ \mbox{ as } \ t\longrightarrow\infty
$$
and
\begin{equation}\label{larget}
t\mathcal{B}^{-1}\left(\frac{1}{t}\right) = \sqrt{2t}\left( 1 + o(1)\right) \ \mbox{ as } \ t\longrightarrow\infty.
\end{equation}
For large values of $s$ (small values of $t$), let $\rho = 1 + s$ and $r = \frac{1}{t}$, then
$$
\rho\ln\rho - \rho + 1 = r.
$$ Let $\rho = e^z$, then
\begin{equation}\label{asymp}
ze^z - r - e^z + 1 = 0.
\end{equation} This implies
\begin{eqnarray*}
(z -1)e^z &=& r -1, \\ (z -1 )e^{z -1} &=& \frac{r - 1}{e}.
\end{eqnarray*}
Let $w := z - 1 \,\,\,\,\,\, v:=  \frac{r - 1}{e}$. Then
\begin{equation}\label{asymp1}
we^w = v.
\end{equation}
The solution of \eqref{asymp1} is given by
$$
w = \ln v - \ln\ln v + \frac{\ln\ln v}{\ln v} + O\left(\left(\frac{\ln\ln v}{\ln v}\right)^2\right)
$$ (see (2.4.10) and the formula following (2.4.3) in \cite{DEB}). So
\begin{eqnarray*}
z &=& 1 + \ln\frac{r - 1}{e} - \ln\ln\frac{r - 1}{e} + \frac{\ln\ln \frac{r- 1}{e}}{\ln \frac{r-1}{e}}\\
&& + \, O\left(\left(\frac{\ln\ln \frac{r- 1}{e}}{\ln \frac{r-1}{e}}\right)^2\right).
\end{eqnarray*}
Since
\begin{eqnarray*}
&&\ln (r-1) = \ln r + O\left(\frac{1}{r}\right),\\&&\ln\left(\ln (r -1) - 1\right) = \ln\ln r + O\left(\frac{1}{\ln r}\right),
\end{eqnarray*}
we get
\begin{eqnarray*}
 z &=& \ln r - \ln\ln r + \frac{\ln\ln r}{\ln r} + O\left(\frac{1}{\ln r}\right)\\ &=&\ln\frac{1}{t} - \ln\ln\frac{1}{t} + \frac{\ln\ln\frac{1}{t}}{\ln\frac{1}{t}} + O\left(\frac{1}{\ln\frac{1}{t}}\right).
\end{eqnarray*}
This implies
$$
\rho = e^z =  \frac{1}{t\ln\frac{1}{t}}\left( 1 +  \frac{\ln\ln \frac{1}{t}}{\ln\frac{1}{t}} +  O\left(\frac{1}{\ln\frac{1}{t}}\right)\right).
$$
Hence
$$
t\mathcal{B}^{-1}\left(\frac{1}{t}\right) = \frac{1}{\ln\frac{1}{t}}\left(1 +  \frac{\ln\ln \frac{1}{t}}{\ln\frac{1}{t}} +  O\left(\frac{1}{\ln\frac{1}{t}}\right)\right)
$$
implying
\begin{equation}\label{smallt}
t\mathcal{B}^{-1}\left(\frac{1}{t}\right) = \frac{1}{\ln\frac{1}{t}}\left( 1 + o(1)\right) \ \mbox{ as } \ t\longrightarrow 0.
\end{equation}
Let
\begin{equation}\label{taut}
\tau := t\mathcal{B}^{-1}\left(\frac{1}{t}\right) .
\end{equation}
Then
\begin{equation}\label{t1}
\ln\frac{1}{t} = \frac{1 + o(1)}{\tau}.
\end{equation}
From
$$
\tau = \frac{1}{\ln\frac{1}{t}}\left(1 +  \frac{\ln\ln \frac{1}{t}}{\ln\frac{1}{t}} +  O\left(\frac{1}{\ln\frac{1}{t}}\right)\right) ,
$$
we get
\begin{equation}\label{t2}
\ln\frac{1}{t} = \frac{1 +  \frac{\ln\ln \frac{1}{t}}{\ln\frac{1}{t}} +  O\left(\frac{1}{\ln\frac{1}{t}}\right)}{\tau}.
\end{equation}
Now \eqref{t1} implies
\begin{eqnarray*}
\ln\frac{1}{t} &=& \frac{1 +  \frac{\ln\frac{1 + o(1)}{\tau}}{1 + o(1)}\tau + O\left( \frac{\tau}{1 + o(1)}\right)}{\tau} \\&=& \frac{1 +  (1 + o(1))\tau \ln \frac{1}{\tau} + O(\tau)}{\tau}\,.
\end{eqnarray*}
Substituting this into \eqref{t2}, one gets
\begin{eqnarray*}
\ln\frac{1}{t} &=& \frac{1 + \frac{\ln\frac{1 + (1 + o(1))\tau\ln\frac{1}{\tau} + O(\tau)}{\tau}}{1 + (1 + o(1))\tau\ln \frac{1}{\tau} + O(\tau)}\tau + O(\tau)}{\tau}\\ &=& \frac{1}{\tau} -  \ln\tau + O(1).
\end{eqnarray*}
Hence
\begin{equation}\label{large}
t = \tau e^{-\frac{1}{\tau}}e^{O(1)}  \ \mbox{ as } \ \tau \longrightarrow 0.
\end{equation}

\section*{Acknowlegments}
The first author is grateful to the Commonwealth Scholarship Commission in the UK, grant UGCA-2013-138, for the funding when he was a PhD student at King's College London.


\begin{thebibliography}{1}
\bibitem{Ad} R.A. Adams,
{\em Sobolev Spaces.} Academic Press, New York, 1975.


\bibitem{BE} A.A. Balinsky and W.D. Evans, {\em Spectral Analysis of Relativistic Operators.} Imperial College Press, London, 2011.

\bibitem{BEL} A.A. Balinsky, W.D. Evans, and R.T. Lewis, {\em The Analysis and Geometry of Hardy's Inequality.} Universitext,
Springer, Cham etc., 2015.

\bibitem{BerShu} F.A. Berezin and M.A. Shubin,
{\em The Schr\"odinger Equation.} Kluwer, Dordrecht etc., 1991.


\bibitem{BL} M.Sh. Birman and  A. Laptev,
The negative discrete spectrum of a two-dimensional Schr\"odinger operator,
{\it Commun. Pure Appl. Math.} \textbf{49}, 9 (1996), 967--997.

\bibitem{BirSol} M.Sh. Birman and M.Z. Solomyak,
{\em Spectral Theory of Self-Adjoint Operators in Hilbert Space.} Kluwer, Dordrecht etc., 1987.

\bibitem{CWH} S.N. Chandler-Wilde and D.P. Hewett, Well-posed PDE and integral equation formulations for scattering by fractal screens,
{\it SIAM J. Math. Anal.} \textbf{50}, 1 (2018), 677--717.

\bibitem{CWHM} S.N. Chandler-Wilde, D.P. Hewett, and A. Moiola, Sobolev spaces on non-Lipschitz subsets of $\mathbb{R}^n$
with application to boundary integral equations on fractal screens, {\it Integral Equations Operator Theory} \textbf{87}, 2 (2017), 179--224.

\bibitem{CWHMB} S.N. Chandler-Wilde, D.P. Hewett, A. Moiola, and J. Besson, Boundary element methods for acoustic scattering by fractal screens,
(arXiv:1909.05547).

\bibitem{DL2} R. Dautray and J.-L. Lions,
{\em Mathematical Analysis and Numerical Methods for Science and Technology.} Vol. 2,  Springer-Verlag, Berlin, 1988.

\bibitem{DL3} R. Dautray and J.-L. Lions,
{\em Mathematical Analysis and Numerical Methods for Science and Technology.} Vol. 3,  Springer-Verlag, Berlin, 1990.


\bibitem{Dav} G. David and S. Semmes, {\em Fractured Fractals and Broken Dreams.} Clarendon Press, Oxford, 1997.

\bibitem{DEB} N.G. De Bruijn, {\em Asymptotic Methods in Analysis.} North-Holland Publishing Company, Amsterdam,  1970.

\bibitem{GSK} B. Ghosh, S.N. Sinha, and M.V. Kartikeyan, {\em Fractal Apertures in Waveguides, Conducting Screens and Cavities.}
Springer, Berlin--Heidelberg, 2014.

\bibitem{Grig} A. Grigor'yan and N. Nadirashvili, Negative eigenvalues of two-dimensional
Schr\"odinger operators, {\it Arch. Ration. Mech. Anal.} \textbf{217} (2015), 975--1028.

\bibitem{Guz} M. de Guzm\'an,
{\em Differentiation of Integrals in $\mathbb{R}^n$. }
Springer, Berlin--Heidelberg--New York, 1975.

\bibitem{HB} D.P. Hewett and J. Bannister, Acoustic scattering by impedance screens with fractal
boundary, {\it Proc. 14th Int. Conf. on Mathematical and Numerical Aspects of Wave Propagation}, Vienna, Austria, 2019, 80--81.

\bibitem{HUT} J.E. Hutchinson, Fractals and self similarity, {\it Indiana University Mathematics Journal} \textbf{30} (1981), 713--747.

\bibitem{Kar} M. Karuhanga, On estimates for the number of negative eigenvalues of two-dimensional Schr\"odinger operators with potentials supported by Lipschitz curves, {\it J. Math. Appl.}, \textbf{456}, 2 (2017), 1365--1379.

\bibitem{MK} M. Karuhanga, Estimates for the number of eigenvalues ot two-dimensional Schr\"odinger operators lying below the essential spectrum,  arXiv:1609.08098, 2016.

\bibitem{KS} M. Karuhanga and E. Shargorodsky, Counting negative eigenvalues of one-dimensional Schr\"odinger operators with singular potentials,
{\it Gulf J. Math.} \textbf{7}, 2 (2019), 5--15.

\bibitem{KMW} N.N. Khuri, A. Martin and T.T. Wu, Bound states in $n$ dimensions
(especially $n = 1$ and $n = 2$), {\it Few Body Syst.} \textbf{31} (2002), 83--89.


\bibitem{KR} M.A. Krasnosel'skii and Ya.B. Rutickii,
{\em Convex Functions and Orlicz Spaces.} P. Noordhoff,
Groningen, 1961.

\bibitem{LN} A. Laptev and Yu. Netrusov,
On the negative eigenvalues of a class of Schr\"odinger operators. In:
V. Buslaev (ed.) et al., {\em Differential Operators and Spectral Theory.
M. Sh. Birman's 70th anniversary collection.} Providence, RI: American Mathematical Society. Transl., Ser. 2, Am. Math. Soc. \textbf{189(41)} (1999), 173--186.

\bibitem{LapSolo} A. Laptev and M. Solomyak,  On spectral estimates for two-dimensional Schr\"odinger operators,
{\it J. Spectr. Theory} \textbf{3}, 4 (2013), 505--515.

\bibitem{Maz} V.G. Maz'ya,
{\em Sobolev Spaces. With Applications to Elliptic Partial Differential Equations.}
Springer, Berlin--Heidelberg, 2011.

\bibitem{MV} S. Molchanov and B. Vainberg,
On negative eigenvalues of low-dimensional Schr\"odinger operators,
(arXiv:1105.0937).

\bibitem{MV1} S. Molchanov and B. Vainberg, Bargmann type estimates of the counting function
for general Schr\"odinger operators, {\it J. Math. Sci.} \textbf{184}, 4 (2012), 457--508.

\bibitem{FM} F. Morgan {\em Geometric Measure Theory}. 2nd ed., Academic Press, New York, 1995.

\bibitem{MW} A. J. Mulholland and A. J. Walker, Piezoelectric ultrasonic transducers with fractal geometry, {\it Fractals} \textbf{19} (2011), 469--479.

\bibitem{Ne1} K. Nesvit, Scattering and diffraction of TM modes on a grating consisting of a finite number of pre-fractal thin impedance strips,
{\it 2013 European Microwave Conference}, 1143--1146

\bibitem{Ne2} K. Nesvit, Discrete mathematical model of wave diffraction on pre-fractal impedance strips. TM mode case,
{\it AIP Conference Proceedings} \textbf{1561}, 219 (2013), 219--223

\bibitem{Ne3} K. Nesvit, Scattering and propagation of the TE/TM waves on pre-fractal impedance grating in numerical results,
{\it The 8th European Conference on Antennas and Propagation (EuCAP 2014)}, 2773--2777.

\bibitem{RR} M.M. Rao and Z.D. Ren,
{\em Theory of Orlicz Spaces.}
Marcel Dekker, New York, 1991.

\bibitem{RSIV} A. Reed and B. Simon,
\textit{Methods of Modern Mathematical Physics. IV. Analysis of operators}, Academic Press, New York, 1978.

\bibitem{Roz} G.V. Rozemblum, The distribution of the discrete spectrum for singular differential operators, {\it Dokl. Akad. Nauk SSSR}
\textbf{202} (1972), 1012--1015.

\bibitem{Eugene1} E. Shargorodsky,
An estimate for the Morse index of a Stokes wave, {\it Arch. Rational Mech. Anal.} \textbf{209}, 1 (2013), 41--59.

\bibitem{Eugene} E. Shargorodsky,
On negative eigenvalues of two-dimensional Schr\"odingers operators, Proceedings LMS, \textbf{108}, 2 (2014), 441--483.

\bibitem{Sol} M. Solomyak,
Piecewise-polynomial approximation of functions from $H\sp \ell((0,1)\sp d)$, $2\ell=d$, and applications to the spectral theory of the Schr\"odinger operator,
{\it Isr. J. Math.}  \textbf{86}, 1-3 (1994), 253--275.

\bibitem{Sol2} M. Solomyak,
On a class of spectral problems on the half-line and their applications to multi-dimensional problems,
{\it J. Spectr. Theory} \textbf{3}, 2 (2013), 215--235.

\bibitem{Stein} E.M. Stein, {\em Singular Integrals and Differentiability Properties of Functions.} Princeton University Press, New Jersey, 1970.


\bibitem{STR} R.S. Strichartz, Fractals in the large, {\it Can. J. Math} \textbf{50}, 3 (1996), 638--65 .

\bibitem{strok} D.W. Stroock, \textit{A Concise Introduction to the Theory of Integration.} Birkh\"auser, Berlin, 1994.

\bibitem{WG} D.H. Werner and S. Ganguly, An overview of fractal antenna engineering research, {\it IEEE Antennas Propag. Mag.} \textbf{45}
(2003), 38--57.

\end{thebibliography}
\end{document}